\documentclass{article}
\usepackage[utf8]{inputenc}
\usepackage[margin = 1in]{geometry}
\usepackage{xparse}
\usepackage{amsmath}
\usepackage{amssymb}
\usepackage{amsfonts}
\usepackage{graphicx}
\usepackage{diagbox}
\usepackage{color}
\usepackage{multicol}
\usepackage{epstopdf}
\usepackage{bm}

\usepackage[labelformat=simple]{subcaption}

\graphicspath{{figures/}}
\usepackage[hidelinks]{hyperref}
\usepackage{amsthm}

\usepackage[format=plain,
            labelfont={it},
            textfont=it]{caption}

\newtheorem{prop}{{\bf Proposition}}
\newtheorem{lemma}{Lemma}
\newtheorem{thm}{Theorem}

\usepackage[title]{appendix}

\renewcommand\theequation{\thesection.\arabic{equation}}

\renewcommand{\theequation}{\arabic{section}.\arabic{equation}}

\title{Localized Pattern Formation and Oscillatory Instabilities in a Three-component Gierer–Meinhardt Model}
\author{Chunyi Gai \thanks{Department of Mathematics and Statistics,
    University of Northern British Columbia, Prince George, B.C.,
    Canada, V2N 4Z9 (corresponding author)},  Fahad Al Saadi \thanks{Department of Systems Engineering, Military Technological College, Muscat, Oman.}} \date{\today}

\numberwithin{equation}{section}
\begin{document}
\maketitle
 
\begin{abstract}
In this paper, we introduce a three-component Gierer-Meinhardt model in the semi-strong interaction regime, characterized by an asymptotically large diffusivity ratio. A key feature of this model is that the interior spike can undergo Hopf bifurcations in both amplitude and position, leading to rich oscillatory dynamics not present in classical two-component systems.  Using asymptotic analysis and numerical path-following, we construct localized spike equilibria and analyze spike nucleation that occurs through slow passage beyond a saddle-node bifurcation. Moreover, stability of spike equilibrium is analyzed by introducing time-scaling parameters, which reveal two distinct mechanisms: amplitude oscillations triggered by large-eigenvalue instabilities and oscillatory spike motion associated with small eigenvalues. Numerical simulations illustrate these dynamics and their transition regimes. This dual mechanism highlights richer spike behavior in three-component systems and suggests several open problems for future study.

\end{abstract}

\section{Introduction}\label{sec1}

Reaction–diffusion (RD) systems have been used as a fundamental framework to model pattern formation in biological, chemical, and physical systems \cite{turing}. A particularly influential example is the Gierer-Meinhardt (GM) model, which is originally proposed to describe the morphogen interactions underlying biological pattern formation \cite{gierer1972theory}.  The classical GM model involves two interacting chemical species: a short-range self-activator and a long-range inhibitor. Through the mechanism of local self-activation and lateral inhibition, this system captures the emergence of stationary spatial structures such as spots, stripes, and labyrinthine patterns \cite{meinhardt2008models,murray2003spatial}.

Over the past decades, the two-component GM model has been extensively analyzed both analytically and numerically. A focus of these work has been on identifying the conditions for Turing instability, the bifurcation structure of patterned solutions \cite{wei2013mathematical}, and the construction of spatially localized patterns with their stability properties \cite{iron2001stability,ward2002existence,wei2007existence}. These studies have provided detailed asymptotic frameworks for understanding how spike equilibria form, interact, and evolve in one- and two-dimensional domains \cite{ kolokolnikov2006stability, kolokolnikov2009existence, ward2003hopf}, and have provided a cornerstone for understanding the formation of self-organized biological patterns.

Despite its success, the two-component formulation can be restrictive for certain biological or chemical systems, where additional species often play important roles as secondary inhibitors or facilitators. 
To capture such complexity, extensions have been proposed to three or more-component GM-type systems \cite{satnoianu2000turing,piskovsky2025turing,wei2008mutually}. The inclusion of an additional component can fundamentally alter dynamical behavior, leading to richer bifurcation structures, oscillatory instabilities, and novel patterns formation regimes that lie beyond the scope of the classical two-component model \cite{xie2021complex, al2024localized, xie2024oscillatory}.

Motivated by understanding how extra components influence the existence and stability of localized structures, in this paper, we analyze a three-component GM model, which is given as follows:

\begin{subequations}\label{Gm}
\begin{alignat}{2}
%\intertext{Activator species ($x\in \mathbb{R}, t>0$)}
\frac{\partial u}{\partial t} & = a-u+\frac{u^3}{w v}+\delta_1 \frac{\partial^2 u}{\partial x^2}, \,\,  x \in (-l,l), t>0, \label{Gm1}\\
%
%\intertext{Inhibitor species ($x\in \mathbb{R}, t>0$).}
\theta\frac{\partial v}{\partial t} & = u^2-b v+D_v\frac{\partial^2 v}{\partial x^2},\,\,  x\in (-l,l), t>0, \label{Gm2}\\
%
%\intertext{Unreactive form of the activator species ($x\in \mathbb{R}, t>0$)}
\tau\frac{\partial w}{\partial t} & = u -cw +\delta_2 \frac{\partial^2 w}{\partial x^2},  \,\,  x \in (-l,l), t>0,  \label{Gm3}\\
\intertext{with Neumann  boundary conditions}
	& \left. u_x\right|_{x=\pm l} = \left. v_x\right|_{x=\pm l} = 
	  \left. w_x\right|_{x=\pm l} =0.  \nonumber
\end{alignat}
\end{subequations}
This model contains an additional reactant $w$
that acts as an inhibitor of $u$ and interacts with $v$ indirectly through its dependence on the activator $u$. The parameters $a, b, c > 0$, and we focus on the semi-strong regime where $\delta_1\ll 1, \delta_2=\delta_1^2$ and $D_v=\mathcal{O}(1)$. In this regime, the activator $u$ diffuses more slowly than $v$, and $w$ exhibits essentially the same behavior as $u$. To account for the possibility that different species evolve on distinct intrinsic timescales, we introduce the time-scaling parameters $\theta$ and $\tau$.  Our goal is to investigate the new dynamics introduced by the extra component $w$. Moreover, we assume a nontrivial background $a>0$ for the activator, which, as we will show below, is essential to generate new spikes.

\subsection{Main Results}

For the three-component GM model (\ref{Gm}), we construct localized spike equilibria in the limit $\delta_1 \to 0$ using asymptotic analysis and numerical-path following methods. In the presence of the nontrivial activator background $a>0$, we show that spike nucleation (or
insertion) can occur as the inhibitor diffusivity $D_v$ decreases, whereby a new spike is generated either at the domain boundaries or from the quiescent background between adjacent spikes.  This process is illustrated in Figure \ref{fig:GM_flexpde_plot}, where initially a interior spike is located at the domain center, then decreasing $D_v$ first triggers the generation of two boundary spikes. With further decrease, new spikes nucleate in between the interior spike and each boundary spike.

\begin{figure}[htbp]
    \centering
    \includegraphics[width=0.8\textwidth]{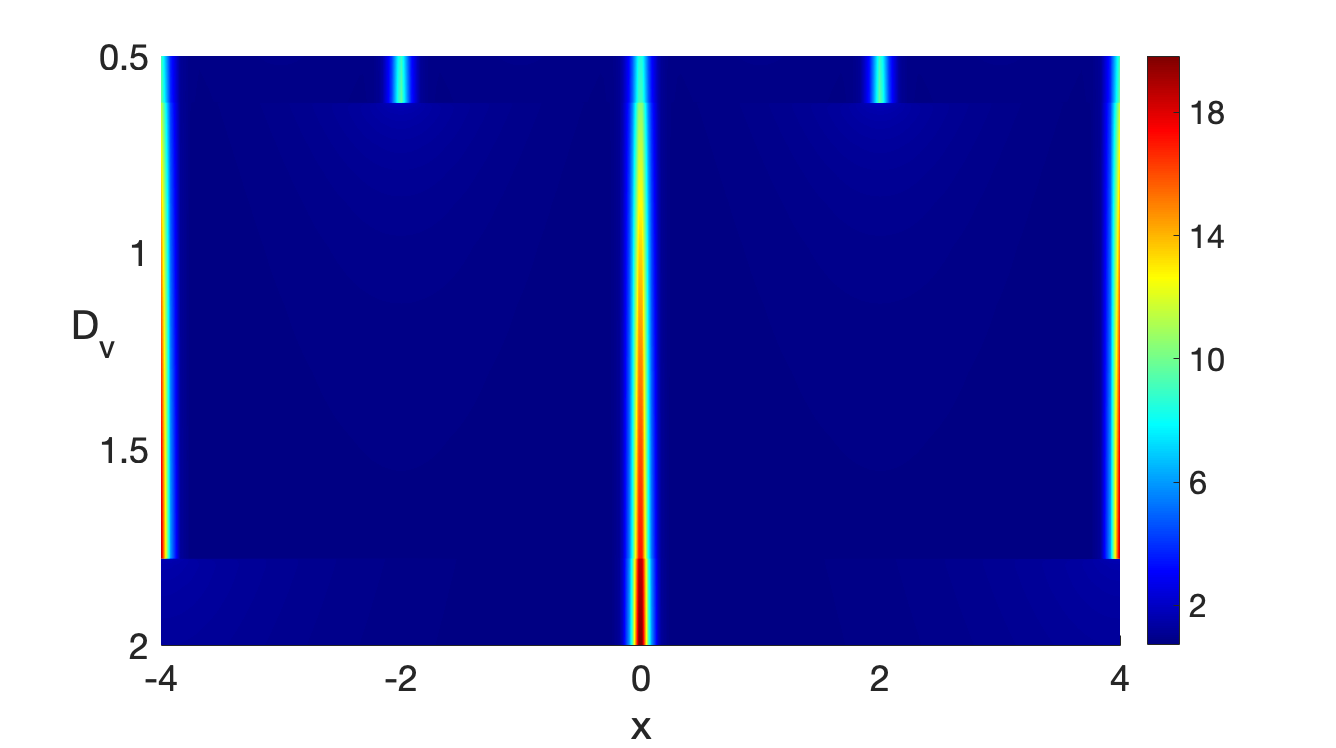}
    \caption{Time-dependent PDE simulations of \eqref{Gm} using \textit{FlexPDE} \cite{flexpde2015} illustrating spike nucleation behavior as the domain half-length $L$ slowly increases as $D_v = 2-1.5*10^{-4}t$. The transitions where boundary spikes first emerge and then later when spikes are nucleated between the interior and boundary spikes. Parameters: $\delta_1 = 0.03^2, \delta_2=\delta_1^2, \theta = \tau =0$, $a =0.7, b=1, c=1$ and $l=4$.}
    \label{fig:GM_flexpde_plot}%
\end{figure}

By relaxing the two time-scaling parameters $\theta$ and $\tau$ in front of the equations of $v$ and $w$, respectively, we investigate the stability of the interior spike equilibria.  In particular,  two classes of eigenvalues are considered. The large eigenvalues that correspond to the structural stability of the pattern are studied by deriving a novel nonlocal eigenvalue problem (NLEP).  The other type are the small eigenvalues that are associated with slow spike motion dynamics. For this instability, we derive asymptotic instability thresholds explicitly, which mark the onset of oscillatory instabilities in the spike motion.

Figure \ref{fig:leig_plot} and \ref{fig:seig_plot} reveal the two qualitatively distinct dynamical behaviors. As shown in Figure \ref{fig:leig_plot}, as $\theta$ increases beyond a critical value, the spike undergoes oscillations in amplitude, leading to instability of the localized structure.  In contrast, in Figure \ref{fig:seig_plot}, by increasing $\tau$ and setting $\theta=0$, the spike experiences an oscillatory instability in position, resulting in periodic motion of the spike across the domain. The oscillatory behavior decays as $\tau$ increases, and and for sufficiently large $\tau,$ the spike ultimately drifts to the boundary and remains there, as shown in Figure \ref{fig:seig_plot} (b).

\begin{figure}
\begin{center}
\includegraphics[width=0.48\textwidth]{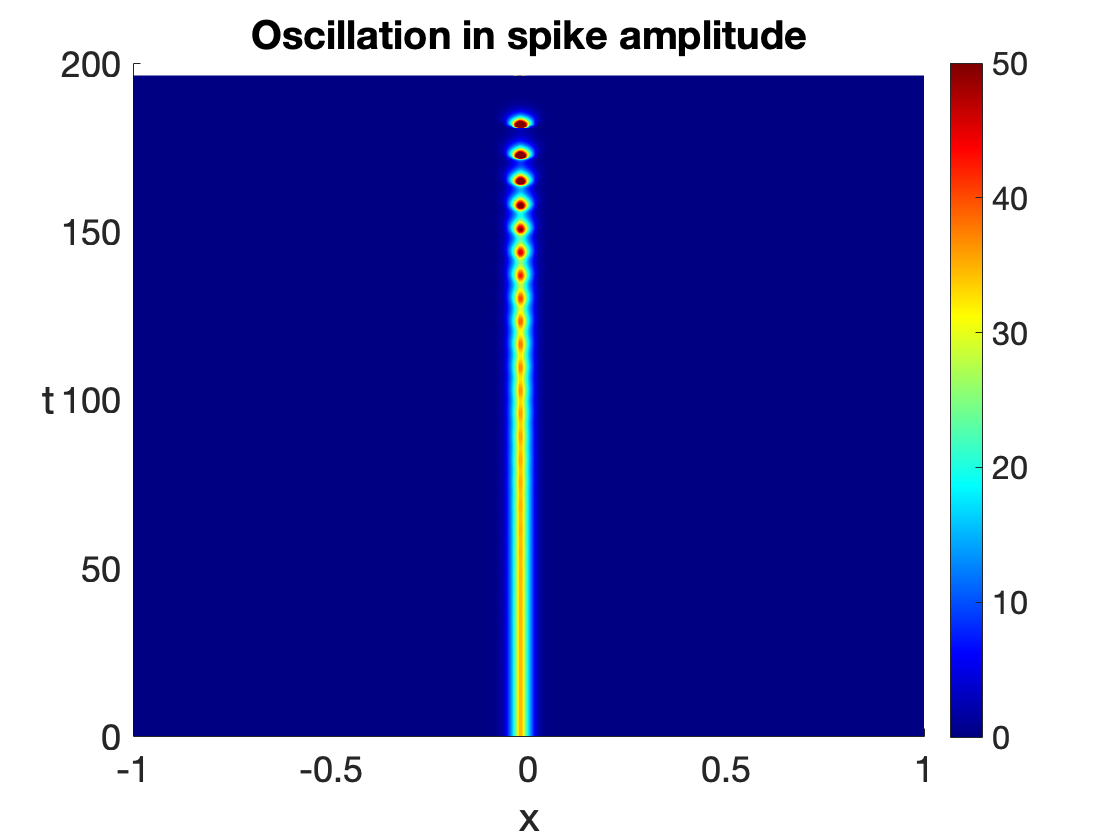} \quad
\includegraphics[width=0.48\textwidth]{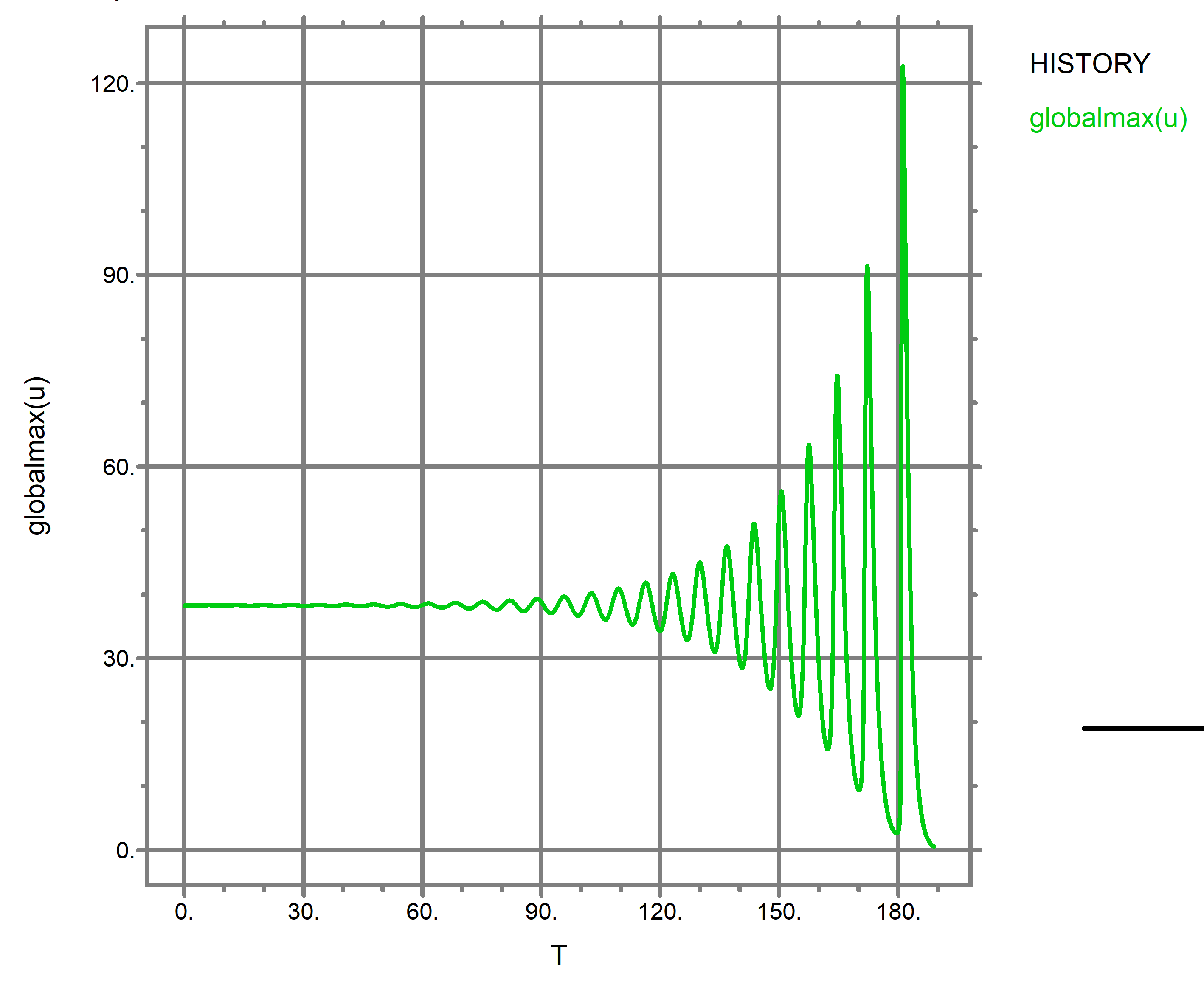} \newline
(a) ~~~~~~~~~~~~~~~~~~~~~~~~~~~~~~~~~~~~~~~~~~~~~~~~~~~~~~~~~~~~ (b)
\caption{Full simulations by {Flexpde} \cite{flexpde2015} illustrating oscillatory behavior in spike amplitudes and spike motion. 
(a) Oscillations in spike amplitude triggered by a large-eigenvalue instability as $\theta=1.5$, which is above the critical value  $\theta_h= 1.34$, the oscillations eventually destroy the spike structure.
(b) The amplitude of $u(x)$ versus time for the interior spike centered at $x = 0$.  Other parameters are: $\delta_1=0.01^2, a=0.01, b=1, c=1, l=1, D_v=1, \tau=0.$}
\label{fig:leig_plot}
\end{center}
\end{figure}

\begin{figure}
\begin{center}
\includegraphics[width=0.48\textwidth]{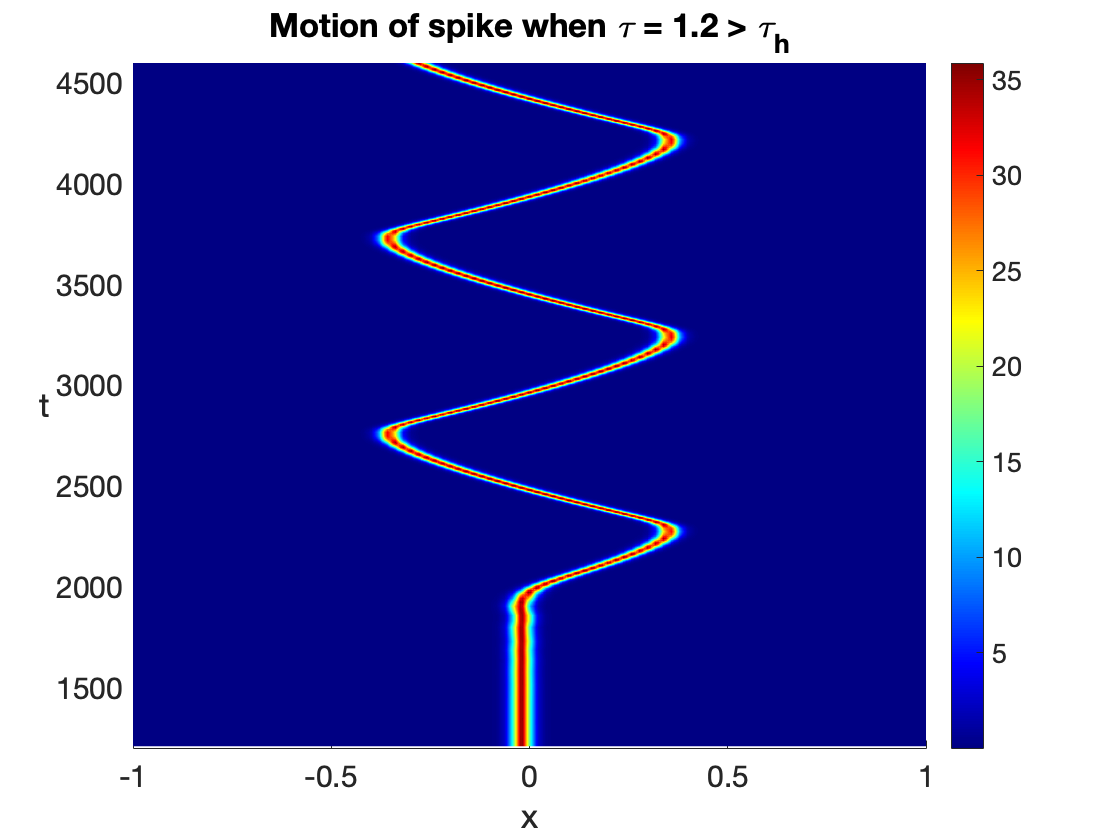} \quad
\includegraphics[width=0.48\textwidth]{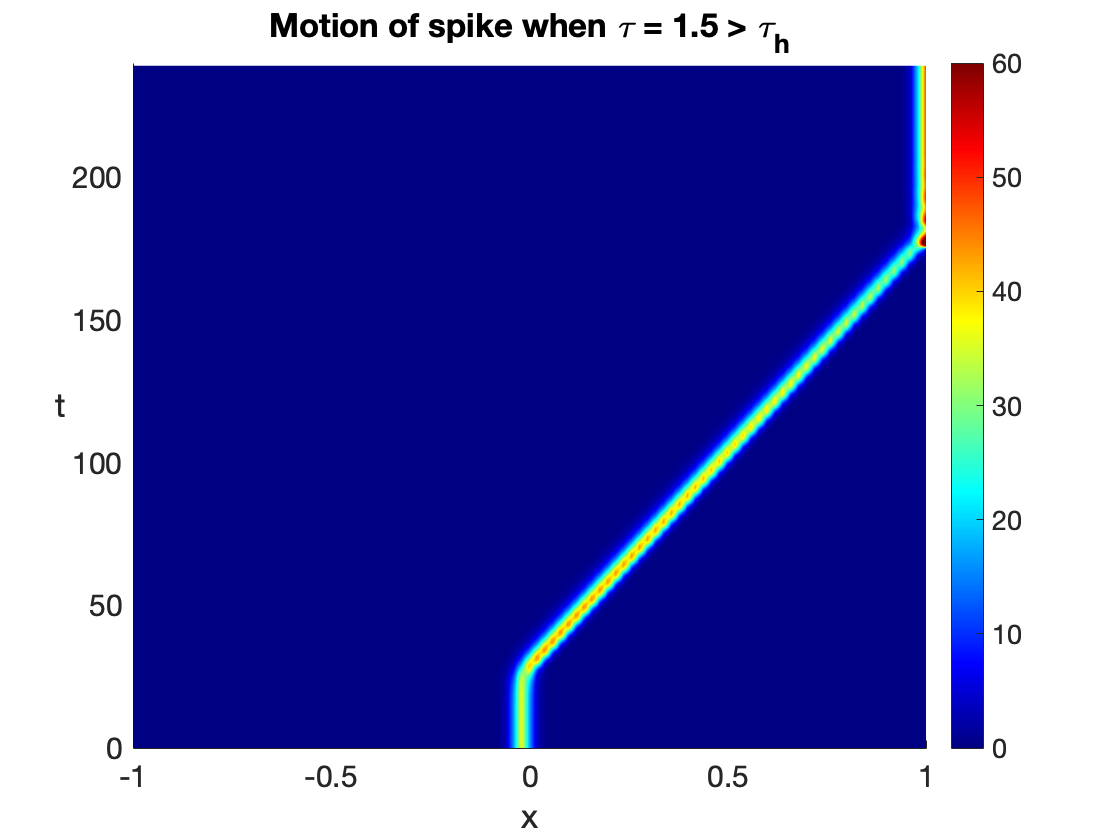} \newline
(a) ~~~~~~~~~~~~~~~~~~~~~~~~~~~~~~~~~~~~~~~~~~~~~~~~~~~~~~~~~~~~ (b)
\caption{Full simulations by {Flexpde} illustrating spike motions for different values of $\tau>\tau_h=1.18$; 
(a) For $\tau=1.2$, the interior spike exhibits small oscillations. (b) For $\tau=1.5$, the spike undergoes larger excursions and eventually drifts toward the boundary. Other parameters are: $\delta_1=0.01^2, D_v=1, a=0.01, b=1, c=1, l=1,\theta=0.$}
\label{fig:seig_plot}
\end{center}
\end{figure}

\subsection{Outline}
The rest of this paper is organized as follows. In Section \ref{Sec2} we use the method of matched asymptotic expansions in the limit $\delta_1\to 0$ to construct quasi-steady state spike solutions for (\ref{Gm}).  We will derive critical values for inhibitor diffusivity $D_v$ at which spike nucleation will occur through slow passage beyond the saddle-node of a nonlinear boundary value problem defined in the outer region away from the core of a spike. In addition, for the regime $a\ll 1$, we obtain an explicit analytical result for the spike profile.

Section \ref{Sec3} $\&$ \ref{Sec4} are devoted to the stability of these spike equilibria in the regime $a\ll 1$. In Section 3, we analyze the spectrum of large eigenvalues by deriving a novel Nonlocal eigenvalue problem (NLEP). Here, three scenarios are considered depending on the value of $\theta$ and $\tau$.   We find that the large eigenvalue instability associated with $\tau$ alone is not observed in full numerical simulations of an interior-spike pattern. Instead, before the large eigenvalue instability is reached, the small eigenvalue becomes unstable and triggers oscillation in slow spike dynamics. 

Finally, Section \ref{Sec5} discusses the implications of our results in relation to existing theory and concludes with directions for future work.

\section{Semi-strong interaction asymptotic analysis} \label{Sec2}
In this section, we apply an asymptotic analysis of semi-strong interactions to construct a one-spike solution centered at $x=0$ to the GM model (\ref{Gm}) defined on the domain $|x|\leq l$, which satisfies the following steady-state problem: 
\begin{subequations}\label{Gmhom}
\begin{alignat}{2}
%\intertext{Activator species ($x\in \mathbb{R}, t>0$)}
 a-u+\frac{u^3}{w v}+\delta_1 \frac{\partial^2 u}{\partial x^2}&=0,\\
 u^2-b v+D_v\frac{\partial^2 v}{\partial x^2}&=0,\\
%
%\intertext{Unreactive form of the activator species ($x\in \mathbb{R}, t>0$)}
 u -cw +\delta_1^2 \frac{\partial^2 w}{\partial x^2}&=0,   
\end{alignat}
\end{subequations}
which are subject to Neumann boundary conditions. 

We will show a novel type of instability that can arise when the system has a nonzero background, that is, $a\neq 0$. This instability is known as spike insertion or nucleation, which involves the emergence of new spikes either near the domain boundaries or at the midpoint between adjacent spikes. For certain ranges of the parameters $a, b, $ and $c$, we show that this instability first occurs as $D_v$ decreases below a
saddle-node point associated with a nonlinear outer problem.

Our analysis follows a similar framework to the recent study \cite{gai2025asymptotic}, which couples a nonlinear inner problem for the spike profile to a nonlinear reduced scalar BVP defined in the outer region away from the spike.

In the mean time, we are interested in the case which arrive at a solution for equation \eqref{Gmhom} that exhibits homoclinic properties in space, before eventually reaching a homogeneous steady state
$$
u \to bc + a, \qquad v\to  \frac{u^2}{b}, \qquad w\to \frac{u}{c} \quad \mbox{as } x \to \pm l.
$$
In this case, however, we will show that, in the presence of the homogeneous steady state, nucleation instabilities do not occur.

\subsection{Asymptotic construction of quasi-equilibria: The inner solution} \label{sec2:inner}
To  construct a quasi steady-state spike solution for \eqref{Gmhom}, in the inner region near $x=0$, we introduce the inner variables $y, U(y), V(y)$ and $W(y)$ by

\begin{equation} \label{sec2_inner_var}
	y= \frac{x}{\sqrt{\delta_1}}, \quad u=\frac{U}{\sqrt{\delta_1}}, \quad v=\frac{V}{\sqrt{\delta_1}}, \quad w=\frac{W}{\sqrt{\delta_1}}, 
\end{equation} 
so that the steady-state problem (\ref{Gmhom}) transforms to
\begin{subequations}\label{sec2_Inner}
	\begin{align}
    &U_{yy} -U+ \frac{U^3}{WV}+a\sqrt{\delta_1}=0,\label{inner_1}\\
    &D_v V_{yy}-b\delta_1V+\sqrt{\delta_1} U^2=0,\label{inner_2}\\
	&\delta_1W_{yy}-cW+U=0\label{inner_3}
	\end{align}
\end{subequations} 
on  $y \geq 0$, with $U_y(0)=0$, $V_y(0)=0$, and $W_y(0)=0$. Upon expanding $U(y), V(y), W(y)$ in the power of $\sqrt{\delta_1}$
\begin{equation} \label{sec2_expansion}
	U(y)=U_0+\sqrt{\delta_1} U_1+O(\delta_1),\quad
	V(y)=V_0+\sqrt{\delta_1} V_1+O(\delta_1),\quad
 W(y)=W_0+\sqrt{\delta_1} W_1+O(\delta_1),
\end{equation} 
and substituting them into (\ref{sec2_Inner}) we find that $W=c^{-1}U$ and $V_0$ is a constant to be determined. Then $V_1$ satisfies
\begin{equation} \label{sec2_V1eqn}
    V_{1yy}=-\frac{U_0^2}{D_v}.
\end{equation}
To reflect the non-zero far-field of the activator, we take $U_0$ to be the homoclinic solution of
\begin{equation}\label{sec2_U0_eqn}
    U_{0yy}-U_0+c\frac{U_0^2}{V_0}+a\sqrt{\delta_1}=0;\quad y\geq0, \quad U_{0y}(0)=0, \quad U_0(0)>0.
\end{equation}
Solving (\ref{sec2_U0_eqn}) yields the homoclinic solution
\begin{equation}\label{sec2_U0}
    U_0=\frac{V_0}{c}\left(w_c(y)+\gamma\right),
\end{equation}
where $w_c(y)$  is the unique solution to
\begin{equation}\label{sec2_w_eqn}
    w_{cyy}-(1-2\gamma)w_c+w_c^2=0;\quad w_{cy}(0)=0, \quad w_c(0)>0 \quad \lim_{y\to\infty}w_c=0.
\end{equation}
Here, $\gamma$ satisfies the quadratic equation
\begin{equation} \label{sec2_gam_eqn}
    \gamma^2-\gamma+\frac{ac\sqrt{\delta_1}}{V_0}=0.
\end{equation}
Since we require $\gamma < \frac{1}{2}$ in (\ref{sec2_w_eqn}), we must take $\gamma$ as the smaller root of (\ref{sec2_gam_eqn}), which is given by
\begin{equation}\label{sec2_gam}
    \gamma=\frac{1-\sqrt{1-\frac{4ac\sqrt{\delta_1}}{V_0}}}{2}.
\end{equation}

Note that for $\delta_1\ll 1$, we have from Taylor expansion that $\gamma \sim \frac{ac\sqrt{\delta_1}}{V_0} +O(\delta_1)$. Then the explicit solution to (\ref{sec2_w_eqn}) is calculated as
\begin{equation}\label{sec2_w}
    w_c(y)=\frac{3}{2}\left(1-2\gamma\right) \text{sech}^2\left(\frac{\sqrt{1-2\gamma}}{2}y\right),
\end{equation}
which, from (\ref{sec2_U0}), determines the homoclinic solution to (\ref{sec2_U0_eqn}) up to a constant $V_0$ to be found. 

We now summarize the result in the inner region as follows:

\begin{prop}
In the limit $\delta_1 \to 0$, system (\ref{Gmhom}) admits a one-spike steady-state solution centered at $x = 0$, the steady state is given by
\begin{equation} \label{sec2_inner_ss}
    u\sim \frac{U_0}{\sqrt{\delta_1}}=\frac{V_0}{c\sqrt{\delta_1}}\left(w_c(y)+\gamma\right),\quad
    v\sim \frac{V_0}{\sqrt{\delta_1}}, \quad w=\frac{u}{c}.
\end{equation}
where $\gamma$ is given by (\ref{sec2_gam}), $w_c(y)$ is given in (\ref{sec2_w}), and $V_0$ is to be determined.
\end{prop}

Next, we determine the far-field behavior of $U$ and $V$.  By letting $y \to \infty$, we use (\ref{sec2_inner_var}), (\ref{sec2_expansion}) and (\ref{sec2_U0}) to conclude that
\begin{equation} \label{sec2_u_far}
    u\sim\frac{V_0\gamma}{c\sqrt{\delta_1}},\quad \text{as} \quad y\to \infty.
\end{equation}
In terms of the far-field behavior of $V$, by substituting (\ref{sec2_U0}) into (\ref{sec2_V1eqn}), we obtain that
\begin{align}    \label{sec2_V1yy}
    V_{1yy}&=\frac{1}{D_v}\left(\frac{V_0}{c}\right)^2\left(w_c+\gamma\right)^2 \nonumber \\ 
    &=-\frac{V_0^2}{D_vc^2}\gamma^2-\frac{V_0^2}{D_vc^2}\left(w_c^2+2\gamma w_c\right). 
\end{align}
Upon integrating (\ref{sec2_V1yy}) using the boundary condition $V_{1y}(0) = 0$, we obtain for any $y > 0$ that
\begin{equation} \label{sec2_V1y}
    V_{1y}=-\frac{V_0^2}{D_vc^2}\gamma^2 y-\frac{V_0^2}{D_vc^2}\left(\int_0^yw_c^2\,ds+2\gamma\int_0^yw_c \,ds\right).
\end{equation}
To determine the limiting behavior as $y \to \infty$, we use (\ref{sec2_w}) to calculate
\begin{equation} \label{sec2_int_w}
    \int_0^{\infty} w_c \,dy=3\sqrt{1-2\gamma}, \quad \int_0^{\infty} w_c^2 \,dy=3\left(1-2\gamma\right)^{3/2}.
\end{equation}
Now using (\ref{sec2_int_w}) in (\ref{sec2_V1y}), we conclude that
\begin{equation} \label{sec2_V1y_far}
    \lim_{y\to\infty} \left(V_{1y}+\frac{V_0^2}{D_vc^2}\gamma^2 y\right) = -\frac{V_0^2}{D_vc^2}\left( 3\left(1-2\gamma\right)^{3/2}+6\gamma\sqrt{1-2\gamma} \right) = -3\frac{V_0^2}{D_vc^2}\sqrt{1-2\gamma}.
\end{equation}
In this way, by using (\ref{sec2_V1y_far}) together with (\ref{sec2_inner_var}) and (\ref{sec2_expansion}) we obtain that $v$ has far-field behavior
 \begin{equation}\label{sec2_v_far}
     v\sim \frac{V_0}{\sqrt{\delta_1}}-3\frac{V_0^2}{D_vc^2}\sqrt{1-2\gamma}y-\frac{V_02^2}{2D_vc^2}\gamma^2y^2, \quad \text{as} \quad y\to \infty.
 \end{equation}
\subsection{Asymptotic construction of quasi-equilibria: The outer solution} \label{sec2:outer}
In this section, we match the far-field behavior of the inner solution with an outer solution valid on $0^+ < |x| < l$, to determine $V_0$ and complete the construction of the spike solution.  As this outer solution is even, we focus on half of the domain, where it satisfies
\begin{subequations}\label{sec2_outer}
\begin{alignat}{2}
 a-u+\frac{u^3}{w v}&=0,  \quad u_x(l)=0, \label{outer_1}\\
D_vv_{xx}+u^2-b v&=0, \quad v_x(l)=0,\label{outer_2}\\
 u -cw &=0, \quad w_x(l)=0. \label{outer_3} 
\end{alignat}
\end{subequations}
Solving $w$ in (\ref{outer_3}) and substitute it into (\ref{outer_1}), we obtain that 
\begin{equation} \label{sec2_v}
    v=\frac{cu^2}{u-a} \quad \text{for} \quad u>a,
\end{equation}
which implies 
\begin{equation} \label{sec2_v_x}
    v_x=-\frac{cu(2a-u)}{(u-a)^2}u_x \quad \text{for} \quad u>a.
\end{equation}
Upon substituting (\ref{sec2_v}) and (\ref{sec2_v_x}) into (\ref{outer_2}) we obtain that the outer problem for $u$ is

\begin{equation}\label{sec2_u_outer_problem}
    D_v\left(f(u)u_x\right)_x=R(u), \quad 0^+<x<l, \quad u_x(l)=0,
\end{equation}
where 
\begin{equation} \label{sec2_f,R}
    f(u)=\frac{cu(2a-u)}{(u-a)^2}, \quad\text{and} \quad R(u)=u^2-b\frac{cu^2}{u-a}.
\end{equation}

The problem (\ref{sec2_u_outer_problem}) is well-posed when $u>0$ and $u_x>0$ on $(0^+,l)$. This implies that we must have $a<u<2a$ on  $(0^+,l)$. 

Next, we derive the matching conditions between the inner and outer solution. From (\ref{sec2_u_far}), together with (\ref{sec2_gam}) for $\gamma$, the first matching condition for the outer solution is 
\begin{equation} \label{sec2_u(0p)}
    u(0^+)=\frac{V_0}{2c\sqrt{\delta_1}}\left(1-\sqrt{1-\frac{4ac\sqrt{\delta_1}}{V_0}}\right).
\end{equation}

For $\frac{4ac\sqrt{\delta_1}}{V_0}<1$, we claim that $u(0^+)>a$. To establish this inequality, we introduce $z=\frac{2ac\sqrt{\delta_1}}{V_0}$ and observe that
\begin{equation}
    \frac{u(0^+)}{a}=\frac{1-\sqrt{1-2z}}{z}.
\end{equation}
Since $\sqrt{1-2z}<1-z$ on $0<z<\frac{1}{2}$, the expression above yields $u(0^+)>a$ whenever $\frac{4ac\sqrt{\delta_1}}{V_0}<1$.

The second matching condition involves matching the flux $u_x$ as $x\to 0^+.$ We first observe that the $\mathcal{O}(y^2)$ term in the far-field behavior (\ref{sec2_v_far}) matched exactly with the quadratic term in the near-field behavior of $u$ as $x\to 0^+$ that arises from the $u^2$ term in (\ref{outer_2}). From (\ref{sec2_v_far}) we conclude that  

\begin{equation} \label{sec2_v_x(0)}
    v_x=-\frac{3V_0^2}{\sqrt{\delta_1}D_vc^2}\sqrt{1-2\gamma} \quad \text{as} \quad x\to 0^+.
\end{equation}

Using (\ref{sec2_v_x}) and (\ref{sec2_v_x(0)}), we obtain the second matching condition for the outer solution in terms of $v$, which is given by
\begin{equation} \label{sec2_matching}
    \lim_{x\to 0^+}-v_x=\lim_{x\to 0^+}f(u)u_x=\frac{3V_0^2}{\sqrt{\delta_1}D_vc^2}\sqrt{1-2\gamma}.
\end{equation}

Next, we establish the following lemma.

\begin{lemma}\label{lemma:R(u)}
    Suppose that $bc > a >0$. Then, on the range of $x$ where $a < u < 2a$, we have $R(u) < 0$
and, consequently, $\frac{du}{dx} > 0$.
\end{lemma}

\begin{proof}
    From (\ref{sec2_f,R}) we observe that
\[ \lim_{x\to a^+}R(x)=-\infty, \quad R(2a)=4a(a-bc). \]
It follows that $R(2a) < 0$ whenever $bc > a $. Moreover, we calculate the derivative 
\begin{equation}\label{Rprime}
    R'(u) = 2u + bf(u),
\end{equation}
where $f(u)$ is defined in (\ref{sec2_f,R}). Since $f(u) > 0$ on $a < u < 2a$, it follows that $R'(u) > 0$ on this interval. Therefore, $R(u)$ is strictly increasing on $a < u < 2a$, and $R(u) < 0$ for $bc > a $.

Moreover, upon integrating (\ref{sec2_u_outer_problem}), and imposing $u_x(l) = 0$, we obtain on $0 < x < l$ that
\begin{equation}\label{sec2_u_outer_integrate}
    D_vf(u)u_x|_x^l=- D_vf(u)u_x=\int_x^l R(u(s))ds<0 
\end{equation}
whenever $a < u < 2a$ and $bc > a $. Since $f(u) > 0$ for $a < u < 2a$, we conclude that $u_x > 0$ on this interval.
\end{proof}

The remaining steps in the analysis to construct the quasi-steady state is to solve (\ref{sec2_u_outer_problem}). We define $\mathcal{G}$ by

\begin{equation}\label{sec2_G_prime}
    \mathcal{G'(\xi)} \equiv -R(\xi) f(\xi)=c(bc+a-\xi)(2a-\xi)\frac{\xi^3}{(\xi-a)^3}>0 \quad \text{on} \quad a<\xi<2a.
\end{equation}
 A first integral of (\ref{sec2_G_prime}) yields
\begin{equation}\label{sec2_G}
    G(\xi)=c\left(\frac{1}{3}(\xi-a)^3+\frac{2a-bc}{2}(\xi-a)^2-2abc(\xi-a)-2a^3\ln(\xi-a)+\frac{a^4-2a^3bc}{\xi-a}-\frac{a^4bc}{2(\xi-a)^2}\right).
\end{equation}

Upon first multiplying (\ref{sec2_u_outer_problem}) by $f(u)u_x$ and then integrating, we use the monotonicity of $u(x)$ and equation (\ref{sec2_G_prime}) to obtain
\begin{equation}\label{sec2_first integral of outer}
    -\frac{1}{2}D_v\left(f(u)u_x\right)^2=\int_{x}^lR(u)f(u)u_x dx dx=-\int_{u(x)}^{\mu}\mathcal{G'(\xi)}d\xi =\mathcal{G}(\mu)-\mathcal{G}(u(x)).
\end{equation}

 Here $\mathcal{G}(\xi)$ is given in (\ref{sec2_G}) and $\mu\equiv u(l)$ satisfies $a < \mu  \leq 2a$. By taking the positive square root in (\ref{sec2_first integral of outer}), we have
\begin{equation} \label{sec2_sqrt_outer}
    f(u)u_x=\sqrt{\frac{2}{D_v}}\sqrt{\mathcal{G}(\mu)-\mathcal{G}(u)}.
\end{equation}

Letting $x\to 0^+$ in (\ref{sec2_sqrt_outer}) and imposing the matching condition (\ref{sec2_matching}) we conclude that $V_0$ is related to $\mu$ by the following nonlinear algebraic equation
\begin{equation}\label{sec2_Newton_1}
    \frac{3V_0^2}{\sqrt{2\delta_1}\sqrt{D_v}c^2}\sqrt{1-2\gamma}=\sqrt{\mathcal{G}(\mu)-\mathcal{G}(u(0^+))},
\end{equation}
where $u(0^+)$ and $\gamma$ are given in terms of $V_0$ by
\begin{equation} \label{sec2_u(0^+)}
    u(0^+)=\frac{V_0}{c\sqrt{\delta_1}}\gamma, \quad \text{where}\quad \gamma=\frac{1-\sqrt{1-\frac{4ac\sqrt{\delta_1}}{V_0}}}{2}.
\end{equation}

Next, upon integrating the separable ODE (\ref{sec2_sqrt_outer}), we get an implicit relation for $u(x)$ on $0^+ < x < l$ given by 
\begin{equation}\label{sec2_out_solution}
    \chi(u(x))=\sqrt{\frac{2}{D_v}}x, \quad \text{where} \quad \chi(u(x))=\int_{u(0^+)}^{u(x)} \frac{f(\xi)}{\sqrt{\mathcal{G}(\mu)-\mathcal{G}(\xi)}}d
    \xi.
\end{equation}

Then, by setting $x = l$ and $\mu\equiv u(l)$ in (\ref{sec2_out_solution}), we obtain an implicit equation for $\mu$ given by

\begin{equation}\label{sec2_implicit_mu}
    \chi(\mu)=\int_{u(0^+)}^{\mu} \frac{f(\xi)}{\sqrt{\mathcal{G}(\mu)-\mathcal{G}(\xi)}}d
    \xi=\sqrt{\frac{2}{D_v}}l,
\end{equation}
with $u(0^+) > a$ given by (\ref{sec2_u(0^+)}).

Since the integral in (\ref{sec2_out_solution}) is improper at $\xi = \mu$, to obtain a more tractable formula for $\chi(\mu)$ we integrate the expression in (\ref{sec2_implicit_mu}) by parts by using $f(\xi) = -\mathcal{G}'(\xi)/R(\xi)$. This yields the proper integral

\begin{equation}\label{sec2_proper_chi}
    \chi(\mu)=-2\frac{\sqrt{\mathcal{G}(\mu)-\mathcal{G}(u(0^+))}}{R(u(0^+))}+2\int_{u(0^+)}^{\mu}\frac{\sqrt{\mathcal{G}(\mu)-\mathcal{G}(\xi)}}{R^2(\xi)}R'(\xi) d\xi,
\end{equation}
where $R(\xi) = \xi^2 - bc\frac{\xi^2}{\xi-a}$ and $\mathcal{G}(\xi)$ is given in (\ref{sec2_G}). On the range $\mu > u(0^+)$, we observe that $\chi(\mu)$ is
positive since $R(\xi) < 0$ and $R'(\xi) > 0$ on $a < \xi < 2a$. Moreover, differentiating (\ref{sec2_proper_chi}) yields 
\begin{equation}\label{sec2_chi_prime}
    \chi'(\mu)= \frac{f(\mu)}{\sqrt{\mathcal{G}(\mu)-\mathcal{G}(u(0^+))}}\frac{R(\mu)}{R(u(0^+))}+\mathcal{G}'(\mu)\int_{u(0^+)}^{\mu}\frac{R'(\xi)}{\sqrt{\mathcal{G}(\mu)-\mathcal{G}(u(0^+))}R^2(\xi)}d\xi.
\end{equation}
Since $f(\xi) > 0, \mathcal{G}'(\xi) > 0$ and $R'(\xi) > 0$ on $a < \xi < 2a$, it follows that $\chi(\mu)$ is a monotonically increasing function of $\mu$ on $u(0^+) < \mu < 2a$ that reaches its maximum value at $\mu = 2a$. As such, recalling that $\mu = u(l)$, we define

\begin{equation}\label{sec2_chi_max}
    \chi_{max} \equiv \chi(2a), \mu_{max} \equiv 2a. 
\end{equation}
We now summarize our asymptotic construction of a one-spike solution to (\ref{Gmhom}) and the mechanism of nucleation instability as follows:
\begin{prop}
     For the regime $a < bc$, the asymptotic construction of a one-spike solution to (\ref{Gmhom}) on $|x| \leq l$ with $\delta_1 \ll 1$ reduces to solving the coupled nonlinear algebraic system (\ref{sec2_Newton_1}) and (\ref{sec2_implicit_mu}) for
$V_0$ and $\mu = u(l)$ in terms of parameters $a, b, c, l, D_v,$ and $\delta_1$.  By solving the two equations simultaneously, we can obtain the value of $V_0$, and complete the spike construction.

Moreover, the one-spike solutions on $-l \leq x \leq l$ undergoes spike insertion (nucleation) as $D_v$ decreases with nonzero background $a\neq 0$. The nucleation threshold $D_{nuc}$ is found by first setting $\mu = \mu_{max}=2a$ in (\ref{sec2_Newton_1}) and solving for $V_0$. This determines $u(0^+)$ from (\ref{sec2_u(0^+)}) as needed in calculating $\chi_{max}$ from (\ref{sec2_chi_max}). In this way, in
terms of $\chi_{max}$ as defined in (\ref{sec2_proper_chi}), spike nucleation for a one-spike solution is predicted to occur when
\begin{equation}\label{nuc_threshold}
    D_v<D_{nuc}\equiv\frac{2l^2}{\chi^2_{max}}.
\end{equation}
\end{prop}

 To validate the analytical spike construction, Figure \ref{fig:GM_V0_newton} compares the asymptotic spike profile with full numerical simulations. In particular, the predicted spike amplitude, obtained by solving the coupled nonlinear system (\ref{sec2_Newton_1}) and (\ref{sec2_implicit_mu}) using Newton’s method, is shown with the numerically computed steady-state solution. The good agreement between theory and computation confirms the accuracy of the asymptotic construction.

\begin{figure}[htbp]
    \centering
    \includegraphics[width=0.6\textwidth]{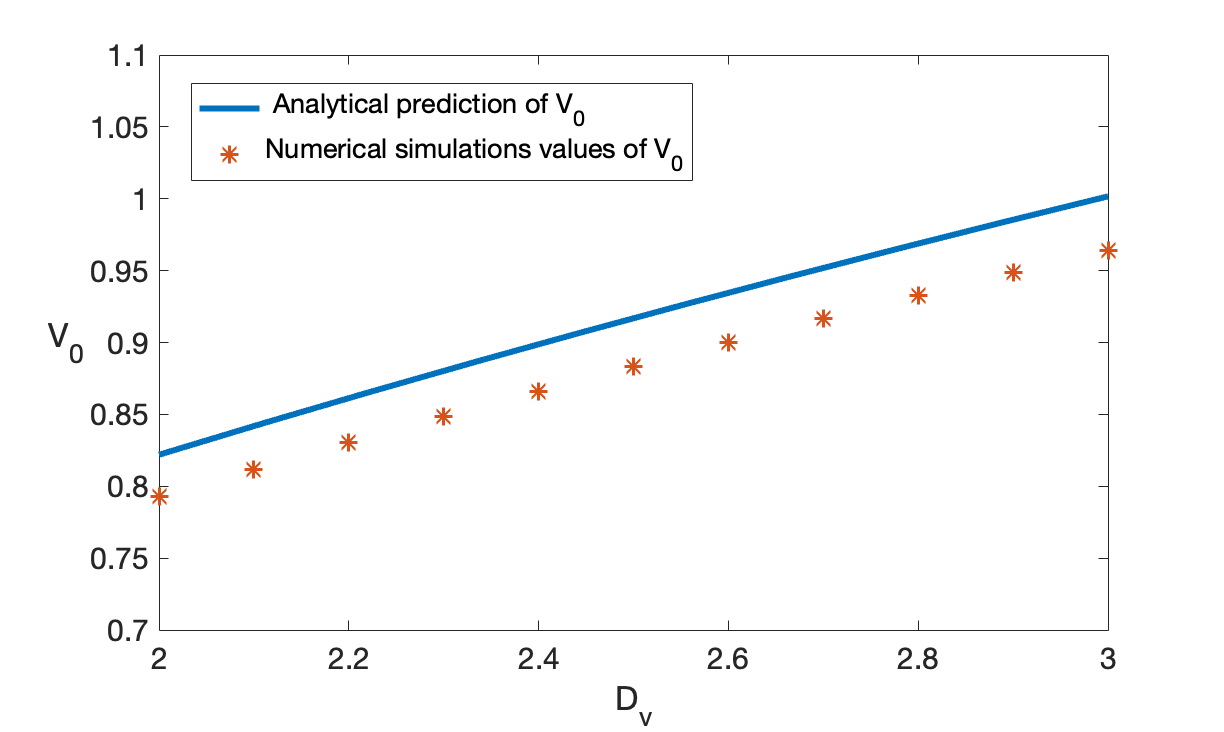}
    \caption{Comparison between analytical and numerical results for $V_0$ as $D_v$ is varied. The solid curve is the analytical result by solving the coupled nonlinear system (\ref{sec2_Newton_1}) and (\ref{sec2_implicit_mu}) using Newton’s method. The red stars are obtained by full simulations of the GM model \eqref{Gm} using \textit{pde2path} \cite{pde2path}. Parameters: $\delta_1 = 0.01^2,a=1,  b=3, c=1$ and $\ell = 3$.}
    \label{fig:GM_V0_newton}%
\end{figure}

In Figure \ref{fig:GM_nuc_D} we compare the asymptotic prediction $D_{nuc}$ in (\ref{nuc_threshold}) for the
critical diffusion threshold of spike nucleation with the numerically computed location of the saddle-node bifurcation point of the full system (\ref{Gm}) using pde2path \cite{pde2path}.  Given the condition that $a<bc=1$, we fix $b=1$ and $c = 1$ in Figure \ref{fig:GM_nuc_D}, and plot $D_{nuc}$ as a function of $a$ for values satisfying $a<1.$ The comparison shows a nice agreement, which verifies the accuracy of the asymptotic prediction.
\begin{figure}[htbp]
    \centering \includegraphics[width=0.55\textwidth, height=6.0cm]{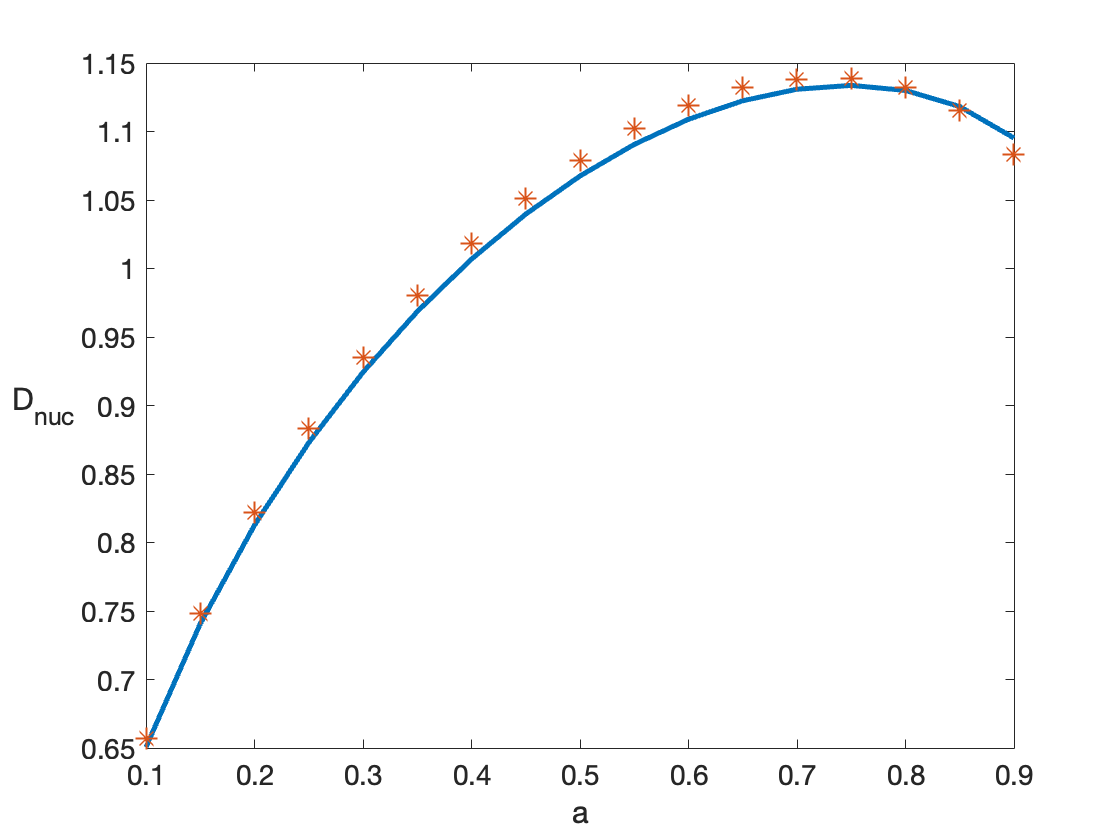}
    \caption{Comparison between asymptotic and numerical results for the nucleation threshold $D_{nuc}$ versus $a$. The solid curves are the asymptotic results given in \eqref{nuc_threshold}. The red stars are the numerical results for the saddle-node bifurcation point as computed from \eqref{Gm} using \textit{pde2path} \cite{pde2path}. Parameters: $\delta_1 = 0.01^2$, $l=4$, $b=1$ and $c = 1$.}
    \label{fig:GM_nuc_D}
\end{figure}

\subsection{Global bifurcation diagram and full PDE simulations}\label{sec2: nuc_bif}

To illustrate the global bifurcation structure, in Figure \ref{fig:GM_nuc_bif}, we consider the parameter set for $a = 0.5, \delta_1 = 0.01^2, b=1, c=1$ and $l=4$, and obtain the numerically computed global bifurcation diagram by path-following the single-spike branch for the GM model (\ref{Gm}) using \textit{pde2path} \cite{flexpde2015}.  The saddle-node bifurcation point at point $(b)$, indicated by the red star, occurs at  $D_{nuc} \approx 1.06$ and marks the onset of spike nucleation.  This value is well-approximated by the asymptotic prediction for $D_{nuc}$ in (\ref{nuc_threshold}) from our asymptotic theory, at which the outer solution ceases to exist (see Figure \ref{fig:GM_nuc_D} for $a = 0.5$). 

This global bifurcation diagram of single-spike steady-states shown in Figure \ref{fig:GM_nuc_bif} provides a detailed view on how new spikes are created near the domain boundaries as $D_v$ approaches the saddle-node bifurcation point.  In the right panel of Figure \ref{fig:GM_nuc_bif} the solution profiles $u(x)$ corresponding to the points marked in the left panel are shown. Starting from the bottom branch at point $(a)$, a single interior spike is present; this state terminates as $D$ decreases below $D_{nuc}$.  Traversing the middle branch from the saddle-node bifurcation, we observe at point $(c)$ that the nucleation of new boundary spikes is fully developed at the domain endpoints.

To further validate these bifurcation scenarios, Figure \ref{fig:GM_flexpde_plot} presents full time-dependent PDE simulations of (\ref{Gm}) computed with FlexPDE \cite{flexpde2015} as the diffusion rate $D_v$  slowly decreases in time by $D = e^{-\rho t}$ with $\rho = 10^{-4}.$ for parameters $a=0.5, \delta_1 = 0.01^2, b=1, c=1$, and $l=4$. As $D_v$ decreases, boundary spikes first appear when 
$D_v\approx 1.07$. Upon further decrease, new spikes nucleate at the midpoint between the interior spike and each boundary spike. These transition values of $D_v$, obtained from the time-dependent simulations, are well-approximated by the critical thresholds predicted asymptotically for the non-existence of the outer solution.

\begin{figure}[htbp]
    \centering
    \includegraphics[width=0.95\textwidth, height=6.0cm]{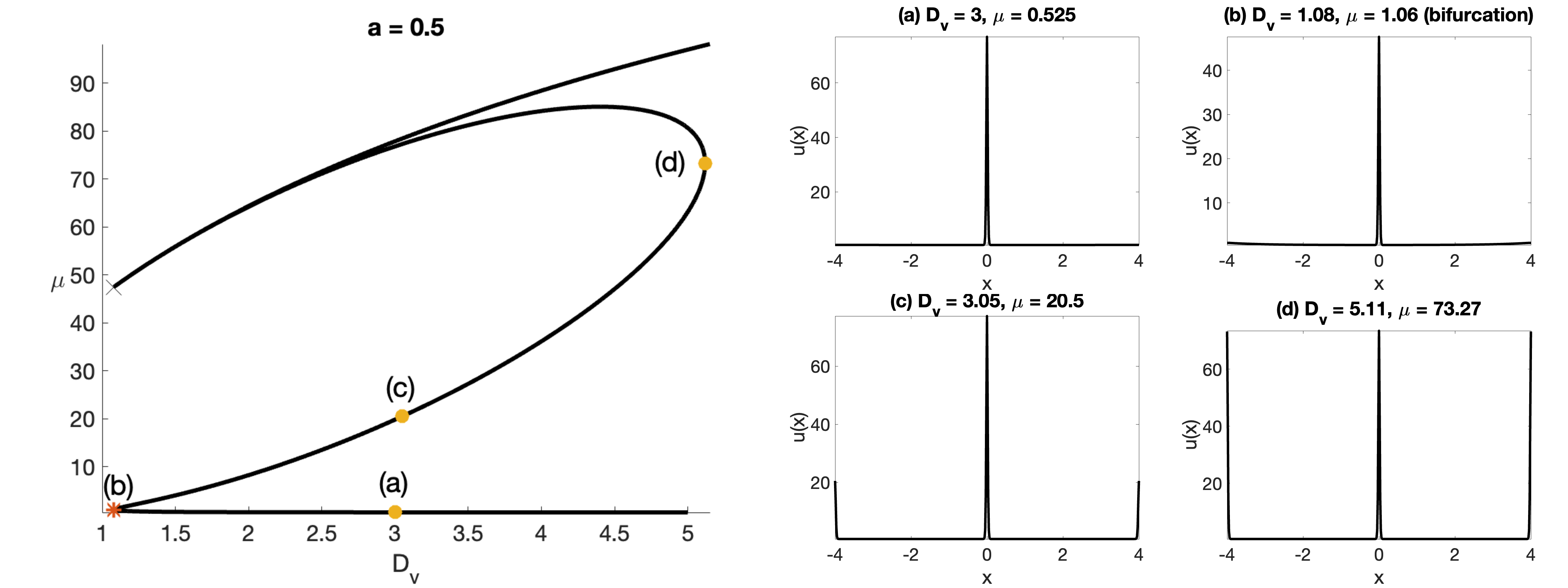}
    \caption{Left panel: Global bifurcation diagram of $\mu=u(l)$ versus $D_v$ for single-spike steady-states for the GM model \eqref{Gm} as computed using \textit{pde2path} \cite{pde2path} for $a = 0.1, b=1, c=1$, $\delta_1 = 0.01^2$ and $l = 4$. Since $a<bc<1$, we predict that spike nucleation occurs as $D_v$ is decreased starting from point (a) on the linearly stable lower branch. The red star is the saddle-node point that signifies the onset of spike nucleation behavior. Right panel: Spike profile $u(x)$ and bifurcation values at the indicated points in the left panel.}
    \label{fig:GM_nuc_bif}%
\end{figure}

\subsection{Asymptotics of the outer solution for \texorpdfstring{$a$}{a} small} \label{Sec2: a small}

In this subsection, we approximate the outer problem (\ref{sec2_u_outer_problem}) by a linear problem when $a\ll 1$. In the limit $a\ll 1$, no spike nucleation behavior occurs.

For $a \ll 1$, we obtain from (\ref{Gmhom} c) that 
$w=\frac{u}{c}$, and plug it in (\ref{Gmhom} a), we get
\begin{equation} \label{sec2_u_out_a_small}
    u\sim a+O(a^2)
\end{equation}
 in the outer region. As a result, from (\ref{Gmhom} b), we obtain that $v(x)$ satisfies
 \begin{equation}\label{sec2_v_eqn_a_small}
     D_vv_{xx}-bv=-a^2+O(a^3), \quad 0^+<x<l, \quad v_x(l)=0,
 \end{equation}
with the matching condition $v(0^+)=\frac{V_0}{\sqrt{\delta_1}}$.
Upon neglecting the $O(a^3)$ term in (\ref{sec2_v_eqn_a_small}), we calculate that 
\begin{equation}\label{sec_2_v_a_small}
    v(x)=\frac{a^2}{b}+\left(\frac{V_0}{\sqrt{\delta_1}}-\frac{a^2}{b}\right)\frac{\cosh\left(\sqrt{b}(l-|x|)/\sqrt{D_v}\right)}{\cosh(\sqrt{b}l/\sqrt{D_v})}.
\end{equation}

By imposing the second matching condition given by (\ref{sec2_v_x(0)}),  we obtain that $V_0$ must satisfy
\begin{equation}\label{sec2_V0_match}
    \frac{3V_0^2}{\sqrt{\delta_1}D_vc^2}\sqrt{1-2\gamma}=\sqrt{\frac{b}{D_v}}\left(\frac{V_0}{\sqrt{\delta_1}}-\frac{a^2}{b}\right)\tanh\left(\sqrt{\frac{b}{D_v}}l\right).
\end{equation}

For $a\ll 1$, (\ref{sec2_u(0^+)}) yields $\gamma\sim \frac{ac\sqrt{\delta_1}}{V_0} \ll 1$. By setting $\gamma\ll 1$ in (\ref{sec2_V0_match}), we get $\sqrt{1-2\gamma}\sim1-\gamma$ and that

\begin{equation}\label{sec2_V0_quadratic}
    3V_0^2-\left(3ac\sqrt{\delta_1}+\sqrt{bD_v}c^2\tanh\left(\sqrt{\frac{b}{D_v}}l\right)\right)V_0+\sqrt{\frac{D_v}{b}}a^2c^2\sqrt{\delta_1}\tanh\left(\sqrt{\frac{b}{D_v}}l\right)=0.
\end{equation}

In the limit of   $\delta_1\to 0$ and $a\ll 1$, we get two asymptotic roots of (\ref{sec2_V0_quadratic})

\begin{equation} \label{sec2_V0_roots_a_small}
    V_{0+}\sim \frac{\sqrt{bD_v}c^2}{3}\tanh\left(\sqrt{\frac{b}{D_v}}l\right), \quad V_{0-}\sim \frac{\sqrt{\delta_1}a^2}{b}. 
\end{equation}

In this way, for $a\ll 1$ our asymptotic result for $v(0) = V_0/\sqrt{\delta_1}$, after using the scaling relation (\ref{sec2_inner_var}), is that

\begin{equation}\label{sec2_v0_a_small}
     v(0)_{+}\sim \frac{\sqrt{bD_v}c^2}{3\sqrt{\delta_1}}\tanh\left(\sqrt{\frac{b}{D_v}}l\right), \quad v(0)_{-}\sim \frac{a^2}{b}.
\end{equation}

In contrast, for $a = \mathcal{O}(1)$, we predict that $v(0)\sim  \frac{V_0}{\sqrt{\delta_1}}$, where $V_0$ must be determined from the coupled
nonlinear algebraic system (\ref{sec2_Newton_1}) and (\ref{sec2_out_solution}). This highlights the qualitative difference between the small 
$a$ and $\mathcal{O}(1)$ regimes:  in the former case a closed-form approximation is available, whereas in the latter case $V_0$ must be computed numerically, for example using Newton’s method.

Figure \ref{fig:GM_V0} compares the asymptotic predictions with full numerical simulations of (\ref{Gm}). In particular, we plot the the spike amplitude $V_0$ by the simple closed-form result in terms of $D_v$ for various values of $a$.   Although (\ref{sec2_V0_roots_a_small}) and (\ref{sec2_v0_a_small}) were derived under the assumption $a\ll 1$. The good agreements observed in Figure \ref{fig:GM_V0} show that the small-$a$ asymptotics remain accurate even for moderately small $a$.

\begin{figure}[htbp]
    \centering
    \includegraphics[width=0.6\textwidth]{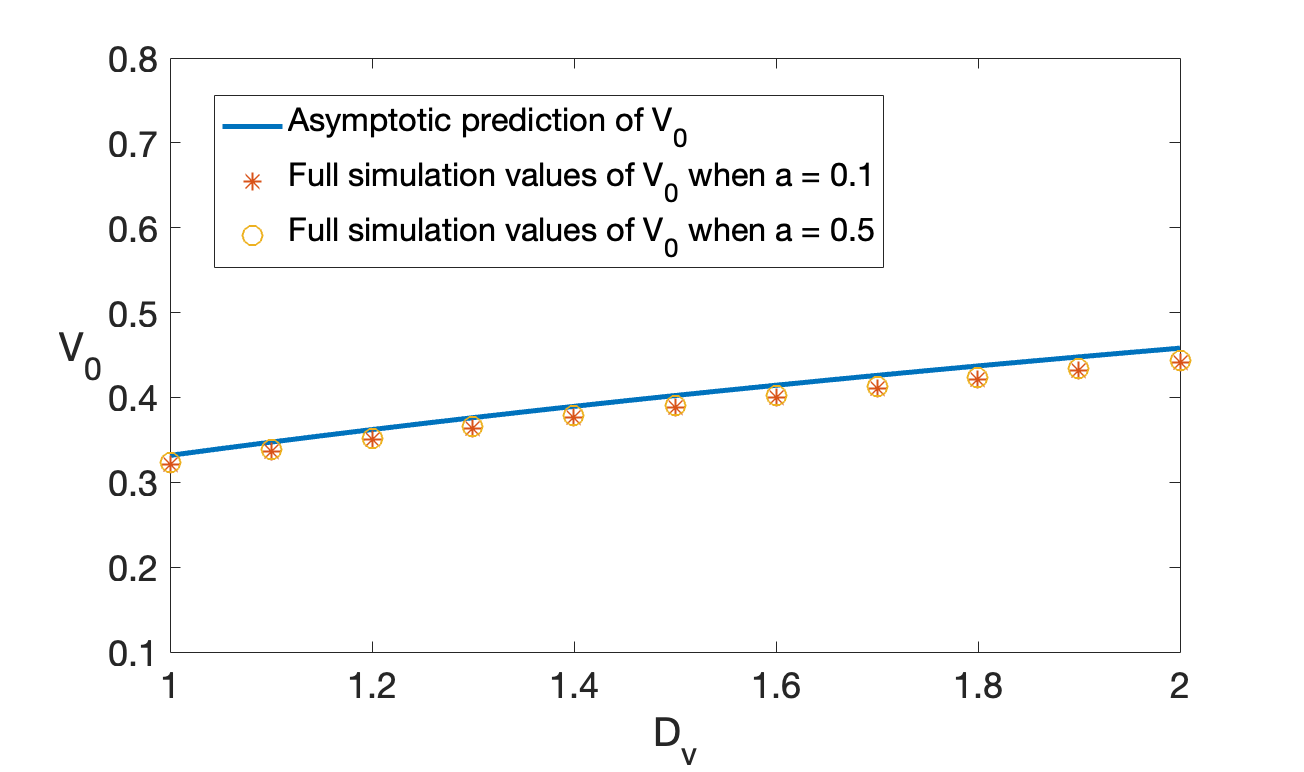}
    \caption{Comparison between asymptotic and numerical results for $V_{0+}$ as $D_L$ is varied. The solid curve is the asymptotic result given in \eqref{sec2_V0_roots_a_small} derived in the limit $a\ll 1$. The red and yellow stars are obtained by full simulations of the GM model \eqref{Gm} using \textit{pde2path} \cite{pde2path} with $a = 0.1$ and $a = 0.5$, respectively. Parameters: $\delta_1 = 0.01^2, b=1, c=1$ and $\ell = 3$.}
    \label{fig:GM_V0}%
\end{figure}

\subsection{Conditions for the absence of nucleation instability}\label{Sec2: no nucleation}

Our analyses in Section \ref{sec2:outer} and \ref{sec2: nuc_bif} have shown that, on certain parameter ranges, spike nucleation behavior can occur as the diffusivity $D_v$ decreases for the GM model in the limit $\delta_1 \ll 1$.  A key mechanism underlying this spike nucleation behavior was that there is a saddle-node bifurcation point at some finite critical value of $D_{nuc}$, beyond which the outer solution ceases to exist.

One key condition for the existence of this saddle node point is that the parameters in the GM model (\ref{Gm}) are such that there is no spatially homogeneous steady-state solution in the range where the outer problem is well defined. More specifically, as shown in the derivation in Lemmas \ref{lemma:R(u)}, we require that in (\ref{sec2_f,R} and (\ref{Rprime})
\[ R(u) < 0  \quad \text{and}  \quad R'(u) < 0,\] 
in the range $a<u<2a$ where the outer problem is well-posed.  Under this condition, the integral $\chi(\mu)$ defined from (\ref{sec2_proper_chi}) increases monotonically in $\mu$ but has finite values as $\mu\to 2a$ below. This limiting value, based on the nonexistence of the outer asymptotic solution, was used in (\ref{nuc_threshold}) to obtain a leading order prediction for the critical threshold of $D_{nuc}$ where a saddle-node point must occur along the single-spike solution branch. The existence of such a saddle-node point is the signature of the onset of spike nucleation behavior.

On the other hand, if the GM system (\ref{Gm}) admits a spatially homogeneous steady state, denoted by $u_{\infty}$, within the well-posedness of the outer problem, it follows that $R(u_{\infty})=0$, with $R'(u_{\infty})>0$. This yields 
\begin{equation}
    u_\infty= a+bc < 2a,
\end{equation}
and the saddle-node bifurcation for spike nucleation does not occur as the threshold $2a$ is not reached.
%To analyze the outer problem for GM model, we need only modify our previous analysis by requiring that $\mu \equiv u(l) < u_\infty$.  More specifically, we observe that $\chi(\mu)$ for (\ref{sec2_proper_chi}) is monotone increasing on $u(0^+)< \mu< u_\infty $ and that $\lim_{\mu\to u_\infty^-}=+\infty$ owing to the non-integrability of the integrals at  $\mu=u_{\infty}$ that define $\chi(\mu)$. As a result, we conclude that for any $D_v > 0$ there is a unique $\mu=\mu^*$ between $(u(0^+), u_\infty)$ at which (\ref{sec2_proper_chi}) has a solution.

In Figure \ref{fig:GM_no_nuc}, we show a global bifurcation diagram obtained from path-following a single-spike steady-state solution of the GM model (\ref{Gm}) using pde2path \cite{pde2path} for the parameters $a = 1.5, b = 1, c=1, \delta_1=0.01^2, D_v = 1.$ Since $bc <a$, and $\mu$ approaches to the spatially homogeneous steady state $u_\infty=a+b=2.5$ but never reaches the limit $2a=3$. Therefore, we observe that no spike nucleation events will occur.

\begin{figure}[htbp]
    \centering
    \includegraphics[width=0.95\textwidth, height=6.0cm]{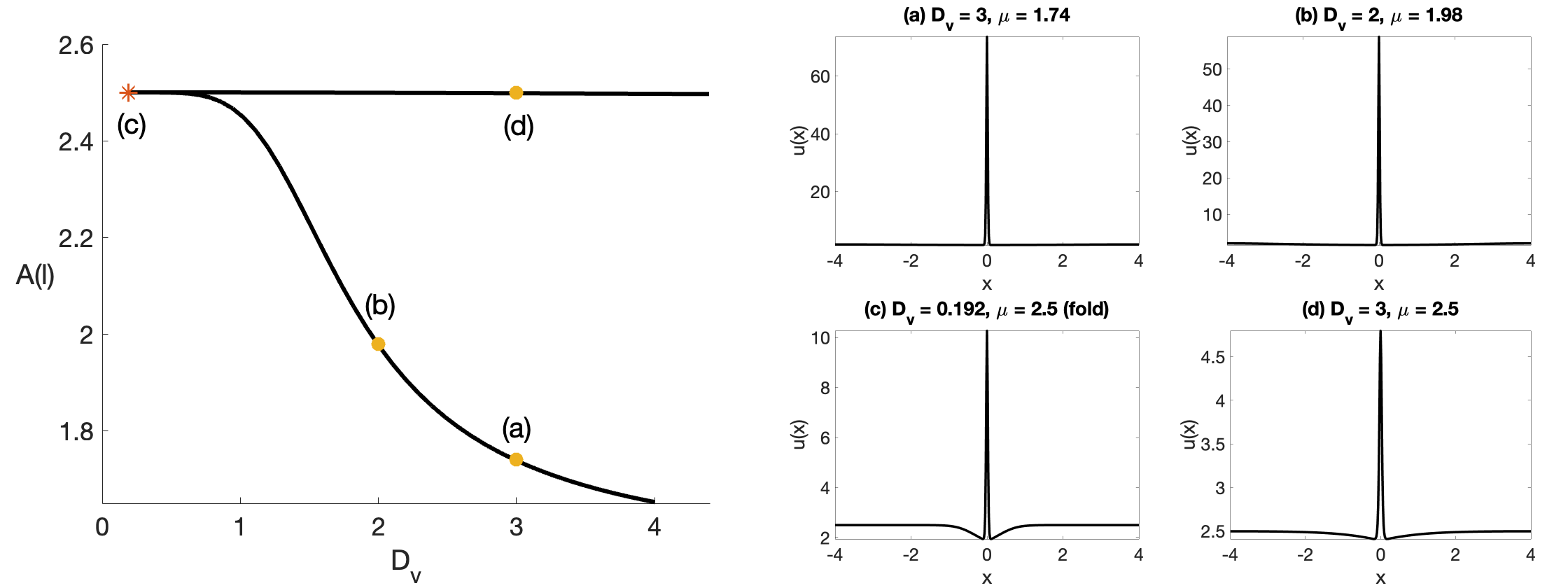}
    \caption{Left panel: Global bifurcation diagram of $\mu=u(l)$ versus $D_v$ for single-spike steady-states for the GM model \eqref{Gm} as computed using \textit{pde2path} \cite{pde2path} for $a = 1.5, b=1, c=1$, $\delta_1 = 0.01^2$ and $l = 4$. Since $a > b$, a spatially uniform steady-state occurs and there is no longer any saddle-node bifurcation on this branch.  Right panel: Spike profile $u(x)$ and bifurcation values at the indicated points in the left panel.}
    \label{fig:GM_no_nuc}%
\end{figure}

In the meantime, we compute the single-spike steady state by modifying the calculation in Section \ref{Sec2: a small} and replacing the outer approximation $u\sim a+O(a^2)$ by $u\sim a+bc$. this leads to the solution for $v(x)$ as
\begin{equation}\label{sec_2_v_no_nuc}
    v(x)=\frac{(a+bc)^2}{b}+\left(\frac{V_0}{\sqrt{\delta_1}}-\frac{(a+bc)^2}{b}\right)\frac{\cosh\left(\sqrt{b}(l-|x|)/\sqrt{D_v}\right)}{\cosh(\sqrt{b}l/\sqrt{D_v})}.
\end{equation}
Then by matching the condition (\ref{sec2_v_x(0)}) as in Section \ref{Sec2: a small}, we obtain two asymptotical roots for $V_0$ that satisfy
\begin{equation}\label{sec2_V0_quadratic_nonuc}
    3V_0^2-\left(3ac\sqrt{\delta_1}+\sqrt{bD_v}c^2\tanh\left(\sqrt{\frac{b}{D_v}}l\right)\right)V_0+\sqrt{\frac{D_v}{b}}(a+bc)^2c^2\sqrt{\delta_1}\tanh\left(\sqrt{\frac{b}{D_v}}l\right)=0.
\end{equation}
In Figure \ref{fig:hetroHome}, we compare the asymptote solution with the full numerical simulation of model \eqref{Gm} obtained using Auto. The two asymptotic values, $v_{0+}=1.47$ and $v_{0-}=0.69$ are obtained from $v_{0\pm}=\frac{V_{0\pm}}{\sqrt{\delta_1}}$ by solving the two roots of (\ref{sec_2_v_no_nuc}).  In the following section, e will show that the upper branch with $v=v_{0+}$ is stable, while the lower branch with $v=v_{0-}$ is unstable, and the asymptotic predictions in Figure \ref{fig:hetroHome} align well with both the stable and unstable branches.

\begin{figure}
	\centering
	\includegraphics[width=\textwidth]{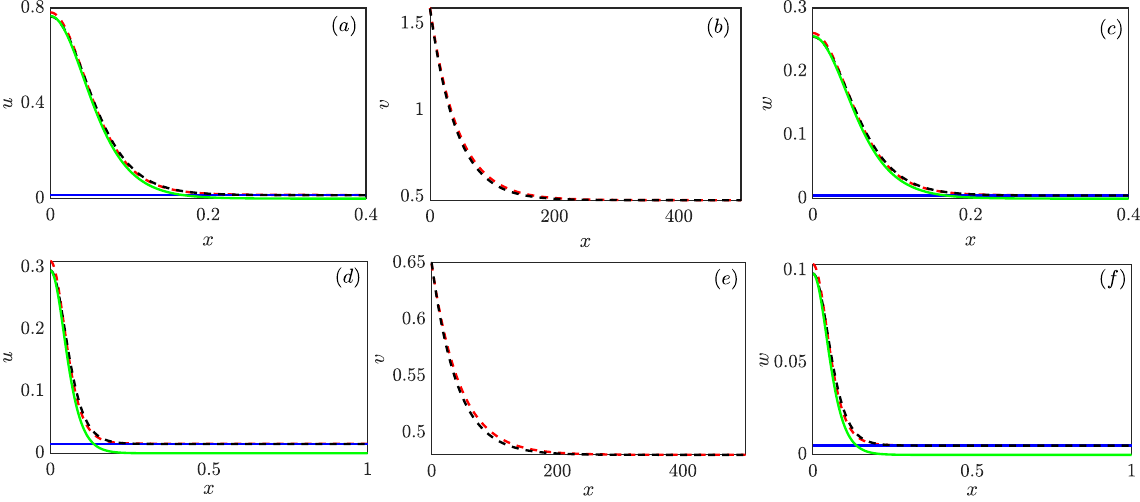}
	\caption{ A comparison of the approximation solution which obtain by the semi-strong analysis with the full numerical solution from AUTO for the upper and lower branch for $\delta_1 = 0.001$, $a = 0.014$, $b = 0.0005$, $c = 3$ and $L=1000$. (a)-(c): The solution of $u$, $v$, and $w$ for the stable branch, respectively, with $v_{0+}=1.47$. The green and blue solid lines present the inner and the outer solution for $u$ and $w$ respectively and the red (black) dashed line indicates the  asymptotic (numerical)  solution. (d)-(f): The approximate and the numerical solution of $u$, $v$ and $w$ in the unstable branch is shown  with $v_{0-}=0.69$ respectively where the colour code is similar to that in the stable branch.}\label{fig:hetroHome}
\end{figure}

\section{Stability of spike equilibrium: Large eigenvalues}\label{Sec3}
In this section we study the linear stability of the one-spike equilibrium obtained in the previous section.  We will consider the ``extended" system obtained by adding two time-scaling parameters, $\theta$ and $\tau$, into system (\ref{Gm}), which takes the form
\begin{subequations}\label{sec3_Gm}
\begin{alignat}{2}
%\intertext{Activator species ($x\in \mathbb{R}, t>0$)}
\frac{\partial u}{\partial t} & = a-u+\frac{u^3}{w v}+\delta_1 \frac{\partial^2 u}{\partial x^2}, \,\,  x \in (-l,l), t>0, \label{sec3_Gm1}\\
%
%\intertext{Inhibitor species ($x\in \mathbb{R}, t>0$).}
\theta \frac{\partial v}{\partial t} & = u^2-b v+D_v\frac{\partial^2 v}{\partial x^2},\,\,  x\in (-l,l), t>0, \label{sec3_Gm2}\\
%
%\intertext{Unreactive form of the activator species ($x\in \mathbb{R}, t>0$)}
\tau \frac{\partial w}{\partial t} & = u -cw +\delta_2 \frac{\partial^2 w}{\partial x^2},  \,\,  x \in (-l,l), t>0,  \label{sec3_Gm3}\\
\intertext{with boundary conditions}
	& \left. u_x\right|_{\pm l} = \left. v_x\right|_{\pm l} = 
	  \left. w_x\right|_{x=\pm l} =0.  \nonumber
\end{alignat}
\end{subequations}

We are particularly interested in the new dynamics introduced by the additional variable $w$, in contrast to previous studies of spike dynamics in two-component RD systems \cite{al2024localized}, where only one type of spike oscillation was observed. In this three-component setting, we demonstrate that two distinct types of spike oscillation can arise. The first corresponds to oscillations in the spike amplitude, associated with the destabilization of a ``large'' eigenvalue, and occurs when $\theta$ increases beyond a critical threshold. The second corresponds to oscillations in the spike position, triggered when a ``small'' eigenvalue becomes unstable. This latter instability is induced by the slower timescale introduced through the variable $w$, and emerges as $\tau$ becomes sufficiently large. These results reveal a richer variety of instabilities than in the classical case of two components, leading to more complex dynamics and even the possibility of chaotic behavior.

In this section, we carry out the large eigenvalue analysis, in which the localized spike equilibrium, denoted by $u_e(x), v_e(x)$ and $w_e(x)$, takes the form given in Sec. \ref{Sec2}. Since $v_e$ is available in closed form only in the regime $a\ll 1$ as shown in Sec. \ref{Sec2: a small}, we have for arbitrary parameter values of $a$ that
\begin{equation}\label{sec3_ss_u,w}
u_e\sim \frac{V_{0}}{c\sqrt{\delta_1}}\left(w_c\left(\frac{x}{\sqrt{\delta_1}}\right)+\gamma\right), w_e=\frac{u_e}{c}.
\end{equation}

By introducing a perturbation around the steady state as
\begin{equation} \label{sec3_perturbations}
    u(x,t)=u_e(x)+\phi(x)e^{\lambda t}, v(x,t)=v_e(x)+\psi(x)e^{\lambda t}, w(x,t)=w_e(x)+\xi(x)e^{\lambda t},
\end{equation}
we obtain the following linearized eigenvalue problem:
\begin{subequations}\label{sec3_3comp-EP}
	\begin{align}
    \lambda \phi &= \delta_1 \phi_{xx} -\phi+ \frac{3u_e^2}{w_e v_e}\phi-\frac{u_e^3}{w_ev_e^2}\psi-\frac{u_e^3}{w_e^2v_e}\xi,\label{3comp-EP_1}\\
    \theta \lambda \psi &=D_v \psi_{xx}-b\psi+2u_e\phi,\label{3comp-EP_2}\\
	\tau \lambda\xi &=\delta_2\xi_{xx}-c\xi+\phi\label{3comp-EP_3}.
	\end{align}
\end{subequations} 
As $\delta_2=\delta_1^2\ll 1$, we get from (\ref{3comp-EP_3}) that 
\begin{equation}\label{sec3_xi}
    \xi=\frac{\phi}{\tau \lambda+c},
\end{equation}
and then (\ref{sec3_3comp-EP}) reduces to 
\begin{subequations}\label{sec3_2comp-EP}
	\begin{align}
    \lambda \phi &= \delta_1 \phi_{xx} -\phi+ \left(2+\frac{\tau\lambda}{\tau \lambda+c}\right)\frac{cu_e}{v_e}\phi-c\frac{u_e^2}{v_e^2}\psi,\label{2comp-EP1}\\
    \theta \lambda \psi &=D_v \psi_{xx}-b\psi+2u_e\phi.\label{2comp-EP2}
	\end{align}
\end{subequations} 
 Introducing the inner variable $x=\sqrt{\delta_1}y$, and using the fact that $\frac{u_e}{v_e}=\frac{w_c+\gamma}{c}$ in the inner region, we obtain $\psi\sim \psi_0$, where $\psi_0$ is a constant to be determined. Then (\ref{2comp-EP1}) reduces to
\begin{equation} \label{sec3_inner_phi}
    \lambda \left(1-\frac{\tau}{c+\tau\lambda}\gamma\right) \phi =  \phi_{yy} -(1-2\gamma)\phi+ \left(2+\frac{\tau\lambda}{c+\tau \lambda}\right)w_c\phi-\frac{(w_c+\gamma)^2}{c}\psi_0.
\end{equation}
To determine the value of $\psi_0$, we consider the outer region, where $\phi$ is localized and can be approximated as a delta function. Therefore, $\psi$ satisfies 
\begin{align}\label{sec3_outer_psi}
    &D_v\psi_{xx}-(b+\theta\lambda)\psi= A \delta(x;0), \quad \psi_x(\pm l)=0, \quad \text{where} \\
    &A=-2\int_{0^-}^{0^+}u_e\phi dx\sim -\frac{2V_{0}}{c}\int_{-\infty}^{\infty}w_c\phi dy.
\end{align}
This implies that 
\begin{equation}\label{sec3_psi}
    \psi =-\frac{A}{b+\theta\lambda}G(x;0),
\end{equation} where $G(x;0)$ satisfies
\begin{equation}\label{sec3_Green's eqn}
    \frac{D_v}{b+\theta \lambda}G_{xx}-G= -\delta(x;0), \quad G_x(\pm l)=0.
\end{equation}
Solving (\ref{sec3_Green's eqn}) yields
\begin{equation} \label{sec3_Green}
    G(x;0)=\frac{\sqrt{b+\theta \lambda}}{2\sqrt{D_v}}\frac{\cosh\left(\sqrt{\frac{b+\theta\lambda}{D_v}}(l-|x|)\right)}{\sinh\left(\sqrt{\frac{b+\theta\lambda}{D_v}}l\right)}.
\end{equation}
By the matching condition $\psi(0)=\psi_0$, we obtain
\begin{equation} \label{sec3_psi0}
    \psi_0=\frac{V_{0}}{c\sqrt{D_v(b+\theta\lambda)}}\frac{1}{\tanh\left(\sqrt{\frac{b+\theta\lambda}{D_v}}l\right)}\int_{-\infty}^{\infty}w_c\phi dy.
\end{equation}

Now substituting (\ref{sec3_psi0}) into (\ref{sec3_inner_phi}) leads to the following non-local eigenvalue problem (NLEP) for arbitrary parameter values of $a$:

\begin{equation}\label{sec3_NLEP_full}
    \lambda \left(1-\frac{\tau}{c+\tau\lambda}\gamma\right) \phi =  \phi_{yy} -(1-2\gamma)\phi+\left(2+\frac{\tau\lambda}{c+\tau\lambda}\right) w_c\phi -\frac{V_{0}}{c^2\sqrt{D_v}\sqrt{b+\theta\lambda}}\frac{(w_c+\gamma)^2}{\tanh\left(\sqrt{\frac{b+\theta\lambda}{D_v}}l\right)}\int_{-\infty}^{\infty}w_c\phi dy.
\end{equation}

Since the NLEP (\ref{sec3_NLEP_full}) is complicated to analyze, in this paper we will focus on the regime at the limit $a \ll 1 $, so that $\gamma\sim \frac{ac\sqrt{\delta_1}}{V_0}\ll 1$, and (\ref{sec3_NLEP_full}) reduced to
\begin{equation}\label{sec3_NLEP}
    \lambda \phi =  \phi_{yy} -\phi+\left(2+\frac{\tau\lambda}{c+\tau\lambda}\right) w_c\phi -\frac{V_{0}}{c^2\sqrt{D_v}\sqrt{b+\theta\lambda}}\frac{w_c^2}{\tanh\left(\sqrt{\frac{b+\theta\lambda}{D_v}}l\right)}\int_{-\infty}^{\infty}w_c\phi dy.
\end{equation}

To analyze the large eigenvalues in the NLEP (\ref{sec3_NLEP}), we treat $\tau$ and $\theta$ as bifurcation parameters and investigate their influence on stability. Specifically,
we consider the following three distinct cases: (1) $\tau = \theta = 0;$ (2)
$\theta > 0, \tau = 0;$ (3) $\theta = 0, \tau > 0$. The stability results in case (1) and (2) in terms of large eigenvalues exhibit a structure similar to that of the NLEP analyzed in \cite{al2024localized}. While in case (3), they extend to a novel eigenvalue problem, which leads to new stability conditions not present in the previous study.

Moreover, in contrast to the results of \cite{al2024localized}, which show that both $\theta$ and $\tau$ can induce Hopf bifurcations in the spike amplitude through large eigenvalue instabilities, we find that the instability associated with $\tau$ in the third component could not be observed in full numerical simulations. Instead, before the large eigenvalue loses stability, a different mechanism is triggered arising from small eigenvalue instabilities. This small-eigenvalue instability and its role in spike dynamics will be studied in detail in the next section.

\subsection{Case 1: \texorpdfstring{$\theta=0, \tau=0$}{theta=0, tau=0}.} \label{sec3.1:large_case1}
We begin with the simplest case by setting $\theta=0, \tau=0$, so that equation (\ref{sec3_NLEP}) reduces to the following NLEP

\begin{equation}\label{sec3_NLEP_c1}
    \lambda \phi(y)=L_0{\phi}-\frac{w_c^2}{A}\int_{-\infty}^{\infty}w_c\phi dy, \quad \text{where} \quad \frac{1}{A}=\frac{V_{0}}{c^2\sqrt{bD_v}}\frac{1}{\tanh\left(\sqrt{\frac{b}{D_v}}l\right)},
\end{equation}
in which $L_0\phi\equiv \phi_{yy}-\phi+2w_c\phi$.  This well-known NLEP was first studied in \cite{wei1999single}. It has the following basic result:

\begin{thm}\label{thm1} (See \cite{wei1999single})
Consider problem (\ref{sec3_NLEP_c1}), let $\lambda$ be an eigenvalue with the largest real part that corresponds to an eigenfunction $\phi$.

\begin{enumerate}
\item If $A>6$, then there exists $\lambda$ with $\lambda>0$.
\item If $A<6$, then either $Re(\lambda)<0$ or $\lambda=0$ with the corresponding eigenfunction $\phi=u_c'(y)$.
\item If $A=6$, then $\lambda=0$ with $\phi=u_c$.
\end{enumerate}
\end{thm}

Now we determine the stability of the two branches of one-spike solution with amplitudes $V_{0-}$ and $V_{0+}$ shown in (\ref{sec2_V0_roots_a_small}).  We substitute $V_{0\pm}$ into (\ref{sec3_NLEP}) to obtain 
\begin{equation}
A(V_{0-})=\frac{c^4b\sqrt{bD_v}}{a^2\sqrt{\delta_1}}.\\
\end{equation}
Since $\delta_1\ll1$, it follows that $A(V_{0-})\gg 6$. By Theorem \ref{thm1} we  conclude that the lower branch of the one-spike steady state with $V_0=V_{0-}$ is unstable. In contrast, for the upper branch with $V_0=V_{0+}$, we have $A(V_{0+})=3<6$, so Theorem \ref{thm1} implies that this branch is stable.  This stability result also explains Figure \ref{fig:GM_V0_newton} and Figure \ref{fig:GM_V0}, where the full simulations agree with the stable branch $V_{0+}$.

\subsection{Case 2: \texorpdfstring{$\theta>0, \tau=0$}{theta>0, tau=0}.} \label{sec3.2:large_case2}
When choosing $\theta>0$ and $\tau=0$, the nonlocal term of (\ref{sec3_NLEP}), with $V_0=V_{0-}$ takes the form
\[A(V_{0-})=\frac{bc^4\sqrt{D_v(b+\theta\lambda)}}{a^2\sqrt{\delta_1}}\frac{\tanh\left(l\sqrt{(\theta\lambda+b)/D_v}\right)}{\tanh\left(l\sqrt{b/D_v} \right)}\gg 1 >6, \quad \text{as} \quad \delta_1 \to 0.\]
Therefore, according to Theorem \ref{thm1}, the lower branch of the one-spike steady state with  $V_0 = V_{0-}$ is unstable.

Now we focus on the stability of the branch of solutions with $V_0=V_{0+}$, and we are interested to investigate how relaxing $\theta$ can destabilize the one-spike equilibrium. With $V_0=V_{0+}$, the NLEP (\ref{sec3_NLEP}) simplifies to
\begin{equation}\label{sec3_NLEP_c2}
\lambda \phi=L_0 \phi-w_c^2\frac{\int w_c \phi  dy}{A(\lambda;\theta)}, \quad\text{where} \quad A(\lambda;\theta)=3\sqrt{1+\frac{\theta \lambda}{b}} \frac{\tanh\left(l\sqrt{(\theta\lambda+b)/D_v}\right)}{\tanh\left(l\sqrt{b/D_v} \right)}.\\
\end{equation}

Equation (\ref{sec3_NLEP_c2}) has a structure similar to the NLEP studied in \cite{alsaadi_2022} and [2023 royal A]. Here we use the same idea to derive the stability result  of (\ref{sec3_NLEP_c2}).   First, we rewrite (\ref{sec3_NLEP_c2}) in the following form
\begin{equation*}
(L_0-\lambda)\phi =w_c^2, \qquad \mathrm{where} \quad
\int w_c \phi \: dy=A(\lambda;\theta)\,, 
\end{equation*}
or 
\begin{equation}\label{sec3_f}
f(\lambda):=\int w_c(L_0-\lambda)^{-1}w_c^2 \: dy =A(\lambda;\theta).
\end{equation}
 The global behavior of the same $f(\lambda)$ was studied in \cite{ward2003hopf}, from which we obtain the following basic results:
 \begin{thm}\label{thm2} (See \cite{ward2003hopf})
$f(\lambda)$ has the behavior
\[f(0)=6, f'(\lambda)>0, f''(\lambda)>0, \lambda\in (0,\frac{5}{4}).\]
Moreover, $f(\lambda)$ has a singularity at $\lambda=\frac{5}{4}$ with $f(\lambda)\to \pm\infty$ as $\lambda\to \frac{5}{4}_{\pm}$. For $\lambda>\frac{5}{4}$, we have $f(\lambda)<0$ and $f(\lambda)\to 0$ as $\lambda\to \infty$.
\end{thm}

The graph of $f(\lambda) $ is shown in Figure \ref{fig:f}(a). We then study the stability of the solution branch with $V_{0}=V_{0+}$. By introducing the rescaled parameter $\hat{\theta}:=\frac{\theta}{b}$, the function $A(\lambda;\theta)$ can be rewritten as

\begin{equation}\label{sec3_A2}
A(\lambda)=3\sqrt{1+\hat{\theta}\lambda}\frac{\tanh\left(l\sqrt{\frac{b}{D_v}}\sqrt{1+\hat{\theta}\lambda}\right)}{\tanh\left(l\sqrt{\frac{b}{D_v}}\right)}. \\
\end{equation}

When $\hat{\theta}$ is sufficiently large, the system can be destabilized via a Hopf bifurcation. This result was first proved in \cite{ward2003hopf}.  Although there is no closed-form expression for the Hopf threshold $\hat{\theta}_h$ at which $Re(\lambda) = 0$, this critical value can be computed numerically by discretizing the NLEP (\ref{sec3_NLEP}) using finite differences. This result is illustrated in Figure \ref{fig:f}(b), where $l, b$ and $c$ are fixed, and the above method is applied to calculate $\hat{\theta}_h$. Close agreement is observed between theoretical predictions and full numerical simulations, with errors of no more than $4.5\%$.  Moreover, we observe that as $D_v$ decreases or $l$ increases, (\ref{sec3_A2}) reduces to $A(\lambda;\theta)\sim 3\sqrt{1+\hat{\theta} \lambda}$, in which case the Hopf threshold approaches the asymptotic value $\hat{\theta}_h\to 2.7492$, as also shown in the figure.

\begin{figure}
\begin{center}
\includegraphics[width=0.45\textwidth]{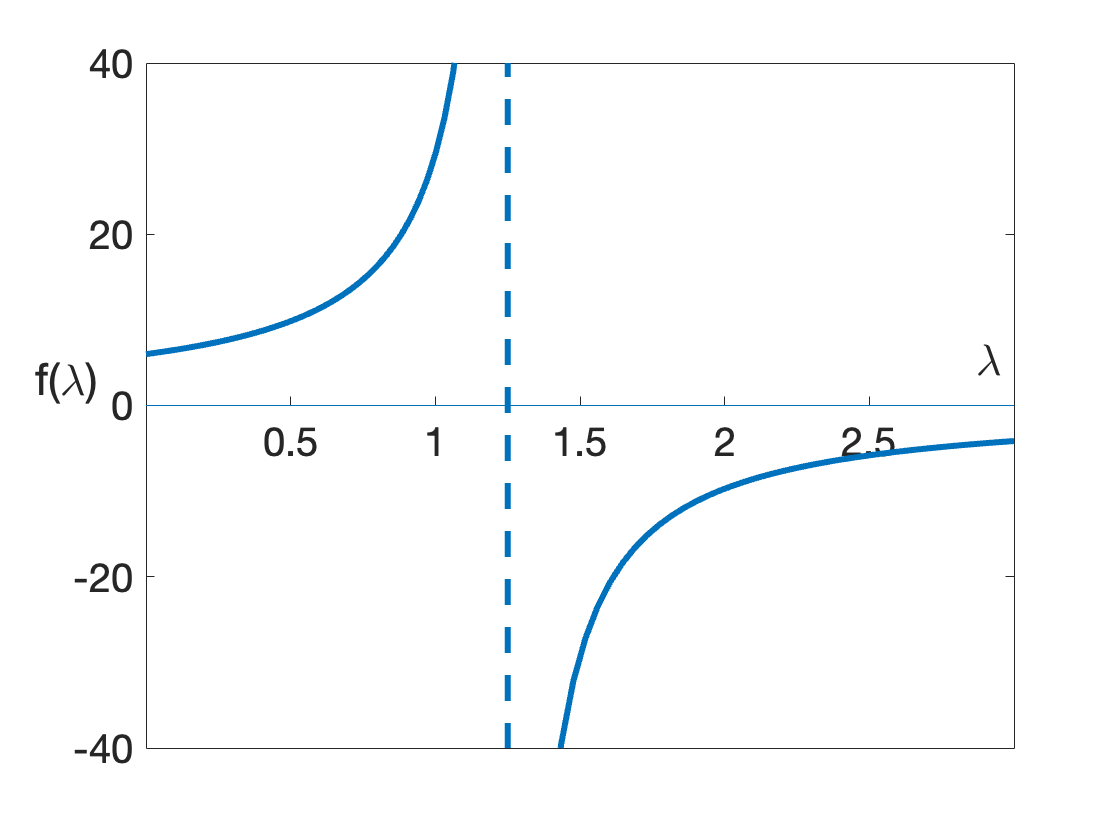} \quad
\includegraphics[width=0.45\textwidth]{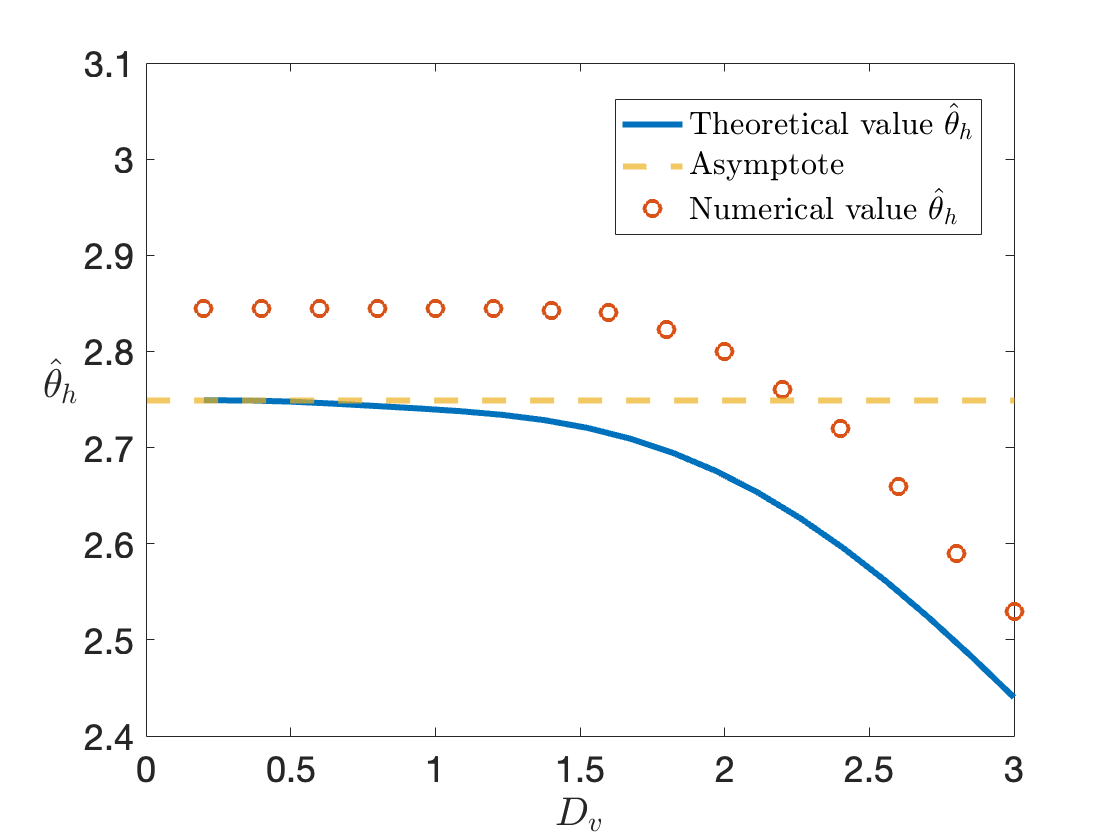}\newline
(a) ~~~~~~~~~~~~~~~~~~~~~~~~~~~~~~~~~~~~~~~~~~~~~~~~~~~~~~~~~~~~ (b)
\caption{Computational results illustrating the stability analysis; 
(a) The function $f(\lambda)$ given in  (\ref{sec3_f}) and it has a singularity at $\lambda=\frac{5}{4}$. (b) Hopf bifurcation points $\hat{\theta}_h$ against $D_v$ with $\tau=0, a=0.01, b=1, c=1, l=3$.}
\label{fig:f}
\end{center}
\end{figure}

\subsection{Case 3: \texorpdfstring{$\theta=0, \tau>0$}{theta=0, tau=0}.} \label{{sec3.3:large_case3}}

In the case where $\theta=0, \tau>0$, upon substituting $V_0=V_{0+}$, NLEP (\ref{sec3_NLEP}) reduces to
\begin{equation}\label{sec3_NLEP_c3}
\lambda \phi=  L_\lambda \phi- \frac{1}{3}w_c^2\int w_c \phi dy,\\
\end{equation}
where 
\begin{equation}\label{sec3_L_lambda}
     L_\lambda \phi= \phi_{yy}-\phi+\left(2+\frac{\tau\lambda}{c+\tau\lambda}\right)w_c\phi.
\end{equation}
Note that in this case, we obtain a new operator $L_\lambda$, with nonlinear dependence on $\lambda$ and the parameters $\tau$ and $c$. In contrast to $L_0$ defined in Sec. \ref{sec3.1:large_case1}, which has a single positive eigenvalue at $\lambda=\frac{5}{4}$, and corresponds to the singularity of $f(\lambda)$ given in (\ref{sec3_f}). Here we investigate the operator $L_\lambda$, by solving the nonlinear local eigenvalue problem 
\begin{equation}\label{sec3_L_lambda_eig}
    L_\lambda\phi=\lambda \phi.
\end{equation}
The problem can be rewritten as a P\"oschl-Teller equation, whose solutions are well studied, and we summarize the result as follows; further details are given in the Appendix.
\begin{thm}\label{sec3_thm_singularity}
    Let $\tau>0$ and $c>0$. For the nonlinear eigenvalue problem given in (\ref{sec3_L_lambda_eig}), there exists a unique positive eigenvalue  $\lambda_0$ that satisfies
    \begin{equation}\label{singularity_lambda}
        \sqrt{1+\lambda_0}=4-3\lambda_0+\frac{3\tau\lambda_0}{c+\tau\lambda_0}.
    \end{equation}
    Ordering the (real) eigenvalues in decreasing order, the \emph{next} eigenvalue is  
    \[\lambda=0,\] and the corresponding eigenfunction $\phi=w'(y)$, with $w(y)=\frac{3}{2}\text{sech}^2\left(\frac{x}{2}\right)$. There are no other nonnegative eigenvalues.
\end{thm}

Now we investigate the nonlocal eigenvalue problem (\ref{sec3_NLEP_c3}) by rewriting it in the following form,
\begin{equation*}
(L_\lambda-\lambda)\phi =w_c^2, \qquad \mathrm{where} \quad
\int w_c \phi \: dy=3\,, 
\end{equation*}
or 
\begin{equation}\label{sec3_g}
g(\lambda;c, \tau):=\int w_c(L_\lambda-\lambda)^{-1}w_c^2 \: dy =3.
\end{equation}
As $L_\lambda$ has a unique positive eigenvalue $\lambda=\lambda_0$, which can be solved via (\ref{singularity_lambda}), the global behavior of $g(\lambda; c, \tau)$ has a structure similar to $f(\lambda)$, but with singularity now depends on $\tau$ and $c$.  In Figure \ref{fig:g} we plot $g(\lambda;c, \tau)$ for fixed $c=0.5$. The right panel shows the singularity curve as $\tau$ increases, in contrast with theorem \ref{thm1}, where the singularity point is fixed.  In particular, as $\tau \to \infty$, the operator approaches $L_\lambda \to \phi_{yy}-\phi+3w_c\phi$, which (see Appendix) has a unique positive eigenvalue $\lambda_\infty=2.5643$. This value is shown as a horizontal asymptote in the figure.

\begin{figure}
\begin{center}
\includegraphics[width=0.45\textwidth]{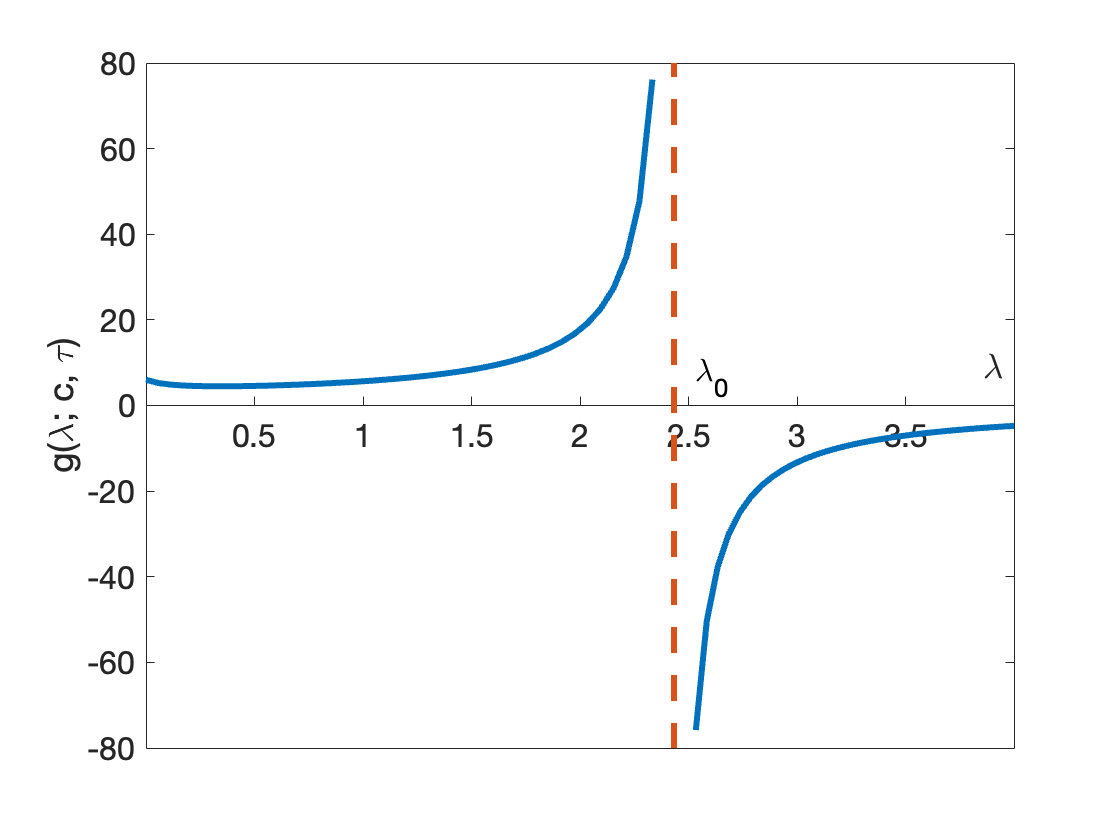} \quad
\includegraphics[width=0.45\textwidth]{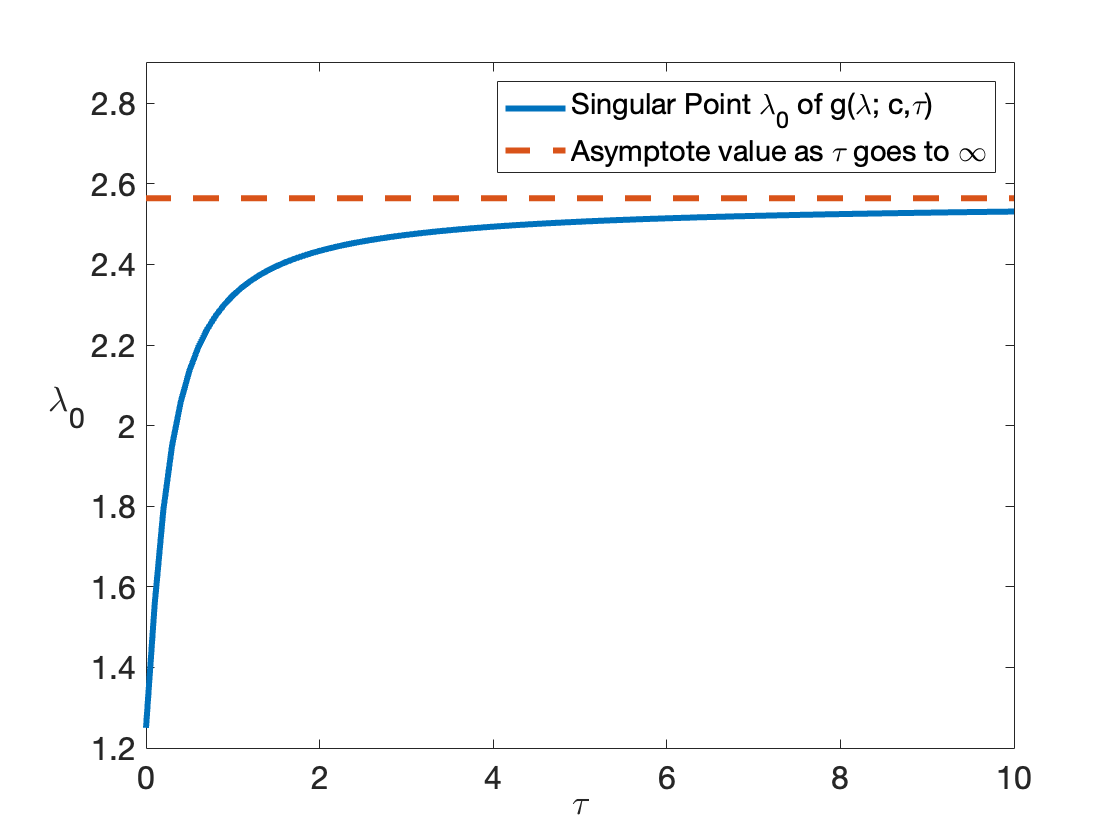}\newline
(a) ~~~~~~~~~~~~~~~~~~~~~~~~~~~~~~~~~~~~~~~~~~~~~~~~~~~~~~~~~~~~ (b)
\caption{Computational results illustrating the global behavior of $g(\lambda; c, \tau)$; 
(a) Function plot of $g(\lambda; c, \tau)$ given in  (\ref{sec3_g}). With $c=0.5$ it has a singularity at $\lambda=\lambda_0\sim 2.43$. (b) Plot of singular point $\lambda_0$ vs. $\tau$, at which $g(\lambda; c, \tau)$ blows up.}
\label{fig:g}
\end{center}
\end{figure}

\begin{figure}
\begin{center}
\includegraphics[width=0.48\textwidth]{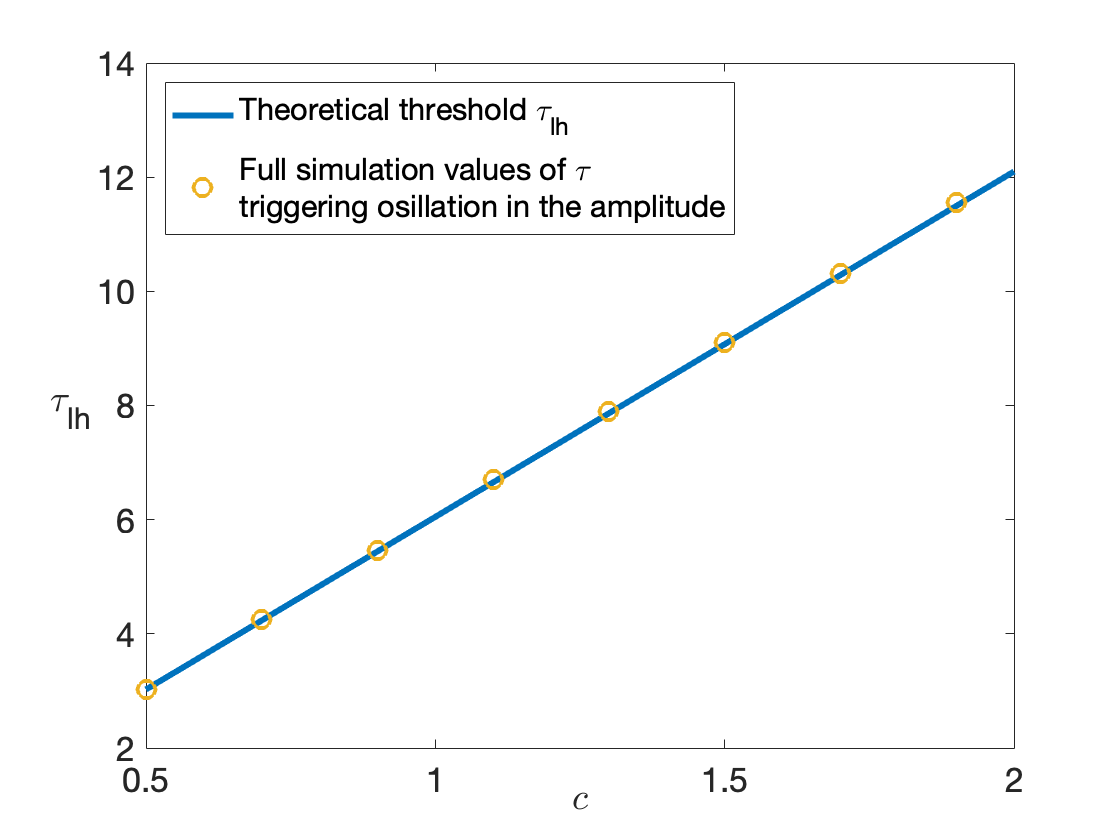} \quad
\includegraphics[width=0.38\textwidth]{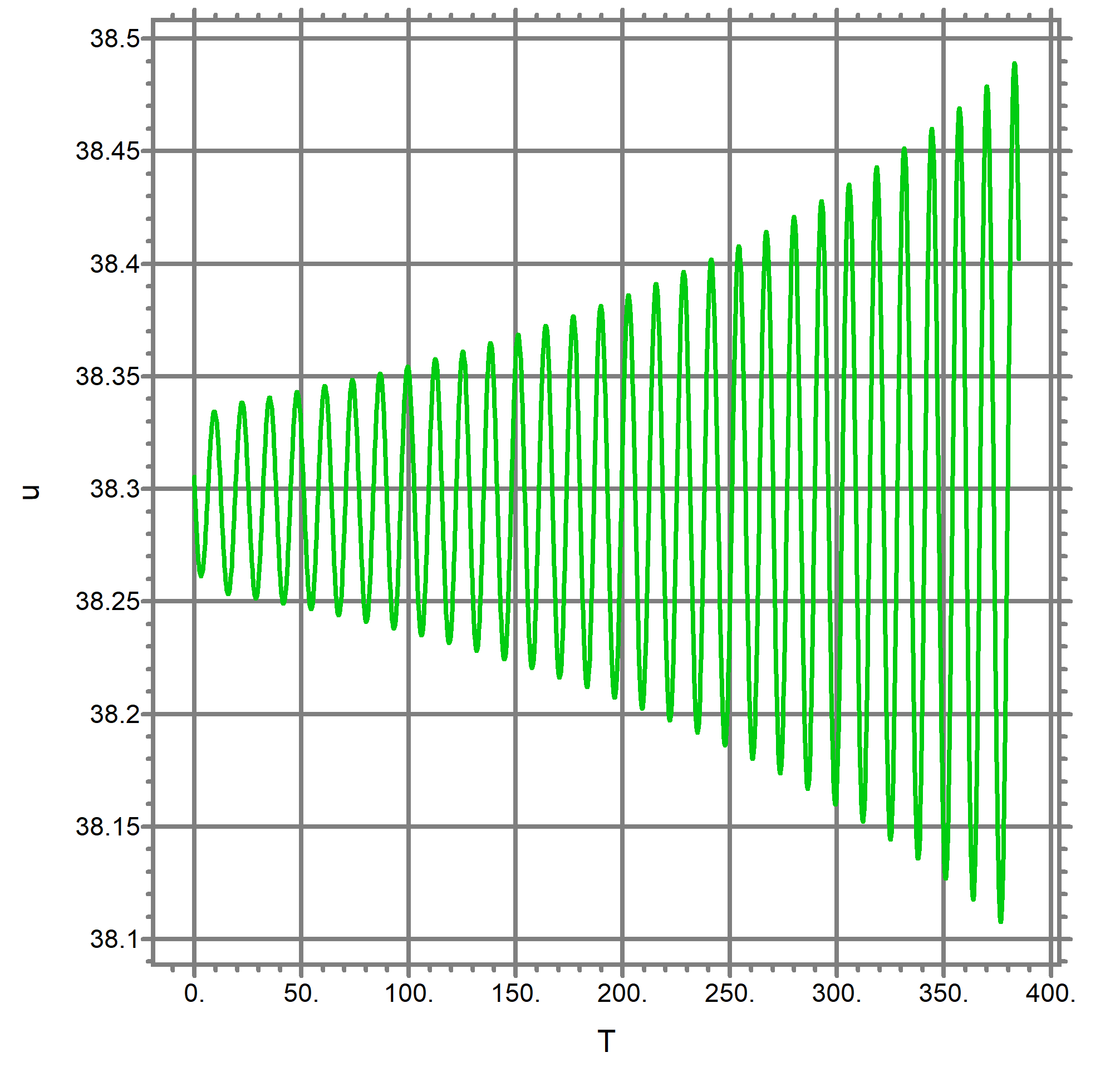}\newline
(a) ~~~~~~~~~~~~~~~~~~~~~~~~~~~~~~~~~~~~~~~~~~~~~~~~~~~~~~~~~~~~ (b)
\caption{(a) Comparison between asymptotic and numerical results for the large eigenvalue threshold $\tau_{lh}$ as parameter $c$ varies. The solid curve is the analytical result solving equation \ref{sec3_g} numerically. The circles denote full simulation results, obtained by recording the point where the boundary half-spike for the GM model \eqref{sec3_Gm} first undergoes a Hopf bifurcation. Parameters: $D_v=1, \theta=0, \delta_1 = 0.01^2, a=0.01,  b=1, \theta=0$ and $l = 1$;
(b) The plot of the spike amplitude at the boundary for $\tau$ just above the threshold $\tau_{lh}$. Here $c=1$, and $\tau=6.2$, slightly above $\tau_{lh}=6.05$.  The other parameters are as panel (a).}
\label{fig:large_tau_threshold}
\end{center}
\end{figure}

When ${\tau}$ is sufficiently large, the single spike steady state loses stability through a Hopf bifurcation at $\tau=\tau_{lh}$. Although no closed-form expression is available for the Hopf threshold $\tau_{lh}$ at which $Re(\lambda) = 0$, we apply similar methods as in Sec. \ref{sec3.2:large_case2} and compute this critical value numerically by discretizing the NLEP (\ref{sec3_NLEP_c3}) using finite differences. The predicted threshold shows excellent agreement with the full time-dependent simulations, as shown in Figure \ref{fig:large_tau_threshold} (a). In the panel (b), we plot the time series of the spike amplitude at the boundary for $\tau$ just above the threshold $\tau_{lh}$.  We also notice that the result $\tau_{lh}$ appears linearly on the parameter $c.$

Note that in the full time-dependent simulations to verify the large-eigenvalue Hopf threshold $\tau_{lh}$ for a single interior spike, simply increasing $\tau$on the symmetric domain $[-l,l]$ does not show the expected oscillations in spike amplitude (as would be associated with a large-eigenvalue Hopf mode).  Instead, we observe oscillations in spike position, indicating that a small-eigenvalue (translational) instability intervenes first. This effect will be analyzed in the next section. Here, to suppress this effect and isolate the large-eigenvalue mechanism, we simulate on the half-domain $[0,l]$ with Neumann boundary conditions, so that a boundary spike at
$x=0$ represents the even extension of an interior spike and the odd translational mode is removed. In this setting, time-dependent PDE simulations performed in FlexPDE yield the threshold shown in \ref{fig:large_tau_threshold}. Moreover, a color plot of the spatiotemporal evolution of the half-spike amplitude as $\tau$ passes $\tau_{lh}$ is shown in Figure \ref{fig:leig_tau10}.

\begin{figure}
\begin{center}
\includegraphics[width=0.48\textwidth]{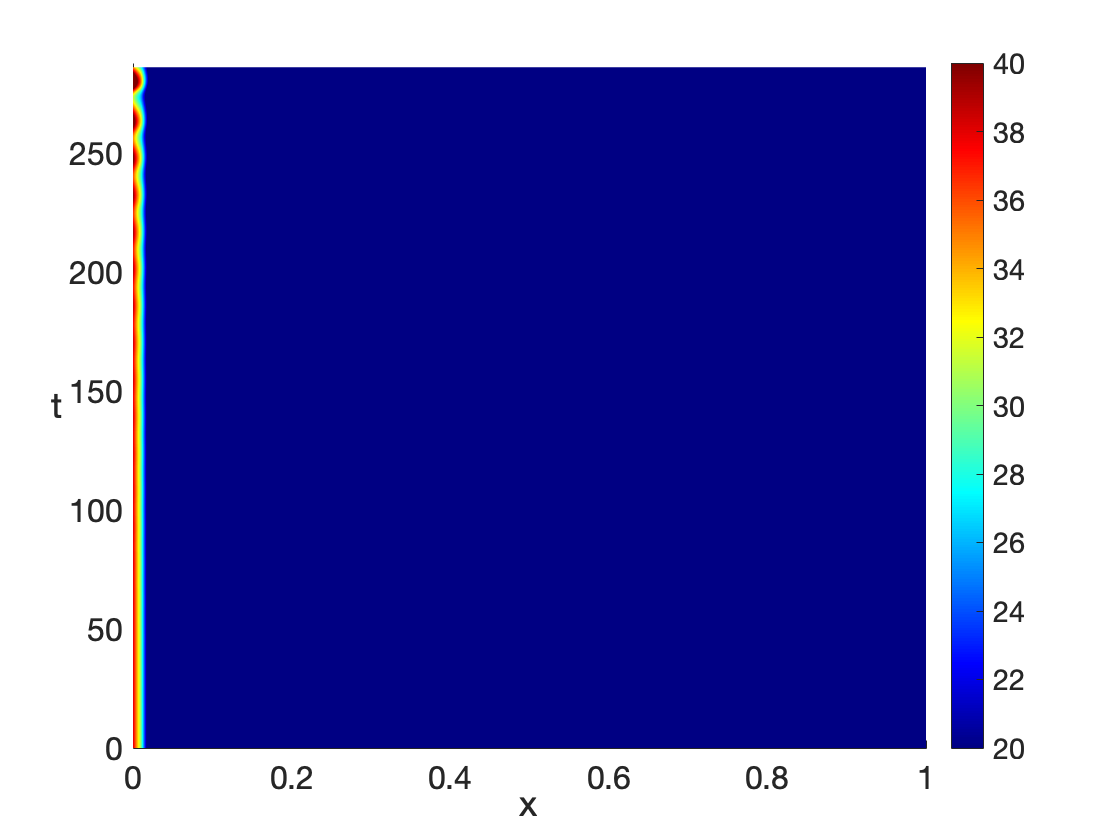} \quad
\includegraphics[width=0.4\textwidth]{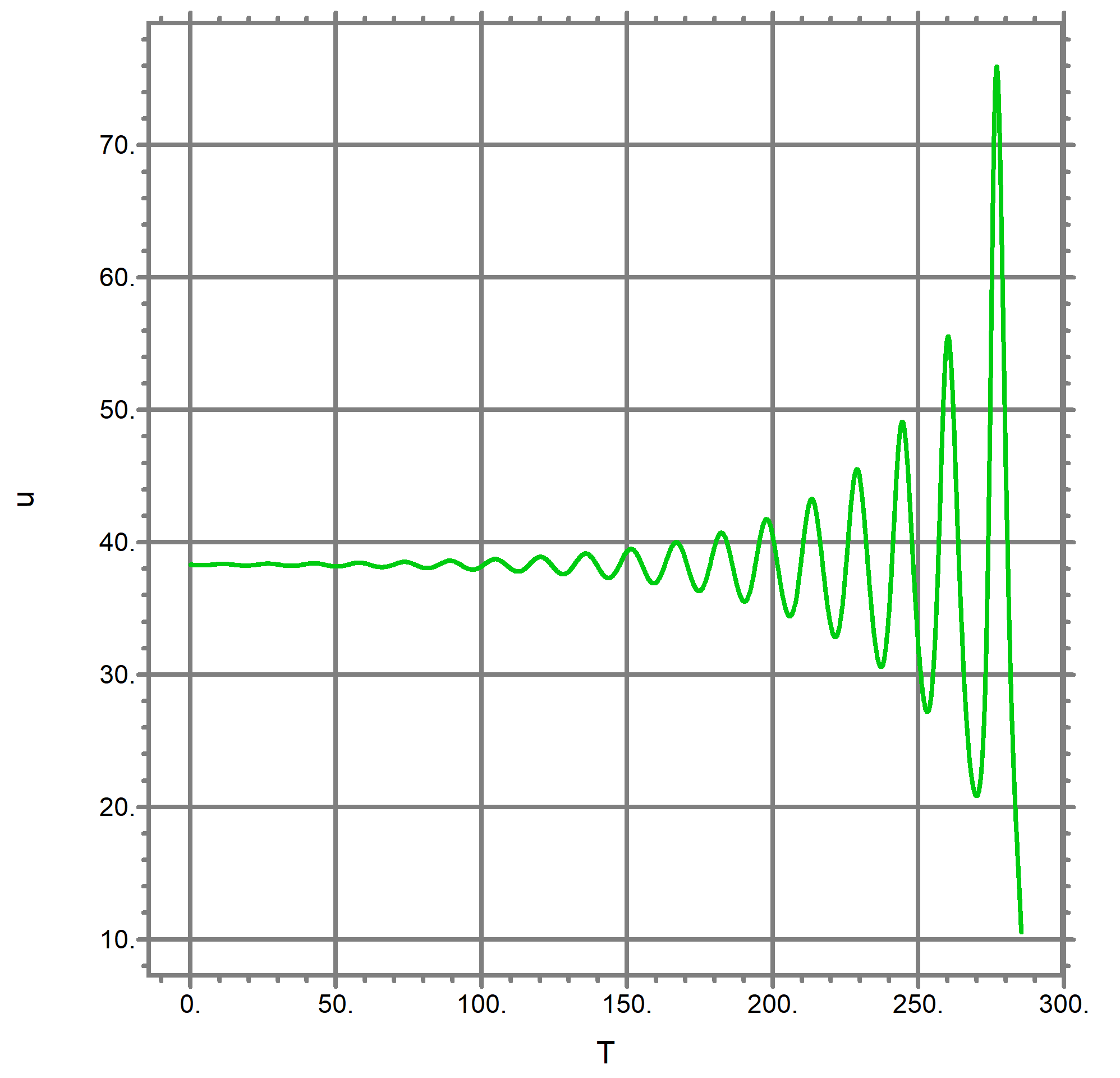}\newline
(a) ~~~~~~~~~~~~~~~~~~~~~~~~~~~~~~~~~~~~~~~~~~~~~~~~~~~~~~~~~~~~ (b)
\caption{Full simulations in FlexPDE illustrating the dynamics of a half-spike under large-eigenvalue instability for $\tau =10>\tau_h=6.05$; (a)For $\tau>\tau_h$, the boundary spike at $x=0$ becomes destabilized by a large-eigenvalue mode and exhibits oscillations in amplitude..; 
(b) The corresponding time series of the spike amplitude at the boundary $x=0$. Other parameters: $\delta_1=0.01^2, \theta=0, D_2=1, a=0.01, b=1; c=1.$}
\label{fig:leig_tau10}
\end{center}
\end{figure}

%Note that when verifying the large-eigenvalue Hopf threshold for a single interior spike in full simulations, increasing $\tau$ on the full symmetric domain $[-l,l]$ does not show sscillation of the spike amplitude as expected correspond to large eigenvalue crossing $0.$ Instead, the spike exhibited oscillations in positions, which indicates that a small-eigenvalue (translational) instability is triggered before the large-eigenvalue mechanism.  To suppress this effect and isolate the large-eigenvalue Hopf mode, we simulate on the half-domain  $[0,l]$ with Neumann boundary conditions, so that a boundary spike at $x=0$ represents the even extension of an interior spike and the odd translational mode is removed.  In this reduced setting, we run time-dependent PDE simulations in FlexPDE and obtained the Figure \ref{fig:large_tau_h_color plot}. 

\section{Stability of spike equilibrium with  \texorpdfstring{$a\ll 1$}{a much less than 1}: Small eigenvalues} \label{Sec4}
In this section, we investigate the stability of the one-spike solution with respect to small eigenvalues.  We will show that the parameter $\tau$ can trigger a small-eigenvalue instability, leading to a Hopf bifurcation associated with oscillatory spike motion. Numerical evidence suggests that this bifurcation is supercritical, giving rise to stable, time-periodic spike dynamics. For simplicity, we will focus on the regime where $\theta=0, \tau>0$ and begin our analysis from the eigenvalue problem (\ref{sec3_3comp-EP}). 

To distinguish the notation from that used in the large-eigenvalue analysis, in terms of $y=\frac{x}{\sqrt{\delta_1}}$, we introduce the following inner variables for the one-spike equilibrium and  perturbations: 
\begin{subequations}\label{sec4_small_eig_var}
\begin{align}
    u_e(x)&=U_e(y), v_e(x)=V_e(y), w_e(x)=W_e(y), \\
    \phi(x)&=\Phi(y), \psi(x)=\Psi(y), \xi(x)=\Xi(y).
\end{align}
\end{subequations}
Then (\ref{sec3_3comp-EP}) becomes
\begin{subequations}\label{sec4_small_3EP}
	\begin{align}
    \lambda \Phi &=  \Phi_{yy} -\Phi+ \frac{3U_e^2}{W_e V_e}\Phi-\frac{U_e^3}{W_eV_e^2}\Psi-\frac{U_e^3}{W_e^2V_e}\Xi,\label{sec4_small_3EP1}\\
    0&=D_v \Psi_{yy}-\delta_1 b\Psi+2U_e\Phi,\label{sec4_small_3EP2}\\
	\tau \lambda\Xi &=\frac{\delta_2}{\delta_1}\Xi_{yy}-c\Xi+\Phi\label{sec4_small_3EP3}.
	\end{align}
\end{subequations} 
Since $\frac{\delta_2}{\delta_1}\ll 1$, we get from (\ref{sec4_small_3EP3}) that $\Xi=\frac{\Phi}{c+\tau\lambda}$, then (\ref{sec4_small_3EP}) reduces to 
\begin{subequations}\label{sec4_small_2EP}
	\begin{align}
    \lambda \Phi &=  \Phi_{yy} -\Phi+ \frac{3U_e^2}{W_e V_e}\Phi-\frac{U_e^3}{W_eV_e^2}\Psi-\frac{U_e^3}{W_e^2V_e}\frac{\Phi}{c+\tau \lambda},\label{sec4_small_2EP1}\\
    0&=D_v \Psi_{yy}-\delta_1 b\Psi+2\delta_1 U_e\Phi,\label{sec4_small_2EP2}
	\end{align}
\end{subequations} 
We expand 
\begin{equation}\label{sec4_expand1}
    \lambda = \delta_1 \lambda+0 + \cdots,\quad \Phi(y)=\Phi_0(y)+\sqrt{\delta_1}\Phi_1(y)+ \cdots, \Psi(y)=\sqrt{\delta_1}\Psi_0(y)+ \cdots, 
\end{equation}
and 
\begin{equation} \label{sec4_expand2}
    U_e(y)=U_{e0}+\sqrt{\delta_1} U_{e1}+\cdots, V_e(y)=V_{e0}+\sqrt{\delta_1} V_{e1}+\cdots, 
    W_e(y)=W_{e0}+\sqrt{\delta_1} W_{e1}+\cdots.
\end{equation}
Substituting (\ref{sec4_expand1}) and (\ref{sec4_expand2}) into (\ref{sec4_small_2EP}) and collecting $\mathcal{O}(1)$ terms,  we get $\Phi_0$ satisfies 
\begin{equation} \label{sec4_Phi_eqn}
    \Phi_{0yy}-\Phi_0+2w_c\Phi_0=0.
\end{equation}
The solution to (\ref{sec4_Phi_eqn}) is
\begin{equation}\label{sec4_Phi0}
    \Phi_0=w_{cy}, \quad \text{where} \quad w_c=\frac{3}{2}\text{sech}^2\left(\frac{y}{2}\right).
\end{equation}

Now collecting the next-order $\mathcal{O}(\varepsilon)$ terms of (\ref{sec4_small_2EP}), and using the fact $\frac{U_{e0}}{V_{e0}}=\frac{w_c}{c}$,  we obtain
\begin{subequations}\label{Phi1_eqn}
	\begin{align}
    \Phi_{1yy} -\Phi_1+ 2w_c\Phi_1&= \frac{w_c^2}{c}\Phi_0+(k_2-k_1)\Phi_0,\label{Phi1_eqn1}\\
    \Psi_{0yy}&=-\frac{2V_0}{cD_v}w_c(y)w_{cy}(y)\label{Phi1_eqn2},
	\end{align}
\end{subequations} 
where $k_1, k_2$ are the $\mathcal{O}(\sqrt{\delta_1})$ coefficients of the two terms $\frac{3U_e^2}{W_e^2 V_e}$ and $\frac{U_e^3}{W_e V_e}$, separately.  Recall the fact that $\mathcal{L}_0w_c=w_c''-w_c+2w_c^2=w_c^2$, so $\Phi_1$ could be written as
\begin{equation}\label{sec4_Phi1}   \Phi_1=\Phi_0(0)w_c+\Phi_{1,\text{odd}},
\end{equation}
where $\Phi_{1,\text{odd}}$ is some odd function.

Next, multiply (\ref{sec4_small_2EP}) by $U_y$ and integrate over $(-\infty, \infty)$, we obtain by integration by parts that
\begin{equation}\label{sec4_int}
   \lambda \int\Phi U_{ey} \left(1-w_c\frac{\tau}{c+\tau \lambda}\right) dy=\int \Phi \left( U_{eyyy}- U_{ey}+2w_c U_{ey}\right) dy-\int\frac{w_c^2}{c}\Psi  U_{ey}dy.
\end{equation}
Here the shorthand notation $\int f dy$ denotes $\int_{-\infty}^{\infty} f(y) dy$.  Moreover, using the fact that $U_e$ satisfies (\ref{inner_1}),  we reduce (\ref{sec4_int}) to the following small eigenvalue problem
\begin{equation}\label{sec4_int2}
   \lambda \int\Phi U_{ey} \left(1-\frac{\tau}{c+\tau\lambda}w_c\right) dy=\int \frac{w_c^2}{c}\Phi V_{ey} dy-\int\frac{w_c^2}{c}\Psi  U_{ey}dy.
\end{equation}

Now, to continue to reduce (\ref{sec4_int2}), we estimate the
left-hand side of (\ref{sec4_int2}) as
\begin{align}\label{sec4_LHS1}
\text{LHS}&=\lambda\frac{V_0}{c\sqrt{\delta_1}}\left(\int w_{cy}^2 dy-\frac{\tau}{c+\tau\lambda}\int w_cw_{cy}^2 dy\right),
\end{align}
and using the leading-order term $\Phi\sim \Phi_0=w_{cy}$, the right-hand side of (\ref{sec4_int2}) becomes
\begin{align}\label{sec4_RHS1}
    \text{RHS}&=\frac{1}{c}\int w_c^2w_{cy}\left(V_{ey}-\frac{1}{c}\frac{V_0}{\sqrt{\delta_1}}\Psi\right) dy \nonumber\\
    &=-\frac{1}{c}\int \frac{ w_c^3}{3}\left(V_{eyy}-\frac{1}{c}\frac{V_0}{\sqrt{\delta_1}}\Psi_y\right) dy,
\end{align}
where from (\ref{inner_2}) $V_{eyy}$ can be expressed as
\begin{equation}\label{sec4_v_eyy}
    V_{eyy}=\frac{\delta_1bV_e-\delta_1U_e^2}{D_v}.
\end{equation}
To estimate the term $\int\frac{w_c^3}{3}\Psi_y$ in (\ref{sec4_RHS1}), we define $F := \int_0^y\frac{w_c^3(s)}{3}ds$ and write
\begin{equation}\label{sec4_rhsterm1}
    \int\frac{w_c^3}{3}\Psi_y=\int\Psi_ydF=\Psi_yF\vert_{-\infty}^\infty-\int F(y)\Psi_{yy}.
\end{equation}
Note that from (\ref{sec4_small_3EP2}), $\Psi_{yy}$ can be approximated as
\begin{equation}\label{sec4_Psi_yy}
    \Psi_{yy}\sim -2\frac{\sqrt{\delta_1V_0}} {cD_v}w_c\Phi_0\nonumber=-2\frac{\sqrt{\delta_1V_0}} {cD_v}w_cw_{cy},
\end{equation}
then applying integration by parts yields
\begin{align}\label{sec4_rhsterm2}
    \int F(y)\Psi_{yy}
    &=-2\int F(y)\frac{\sqrt{\delta_1}V_0} {cD_v}w_cw_{cy}\nonumber\\
    &=\frac{\sqrt{\delta_1}V_0}{cD_v}\int \frac{w_c^5}{3}dy.
\end{align}
Substituting (\ref{sec4_rhsterm2}) into (\ref{sec4_rhsterm1}), we obtain
\begin{equation}\label{sec4_rhsterm3}
    \ int\frac{w_c^3}{3}\Psi_y= 
 \int\frac{w_c^3}{3}\left(\langle \Psi_y \rangle	- \frac{\sqrt{\delta_1}V_0}{cD_v}w_c^2\right)  dy,
\end{equation}
where $\langle \Psi_y \rangle$ is defined as $\langle \Psi_y \rangle :=
 \frac{\Psi_y(\infty)+\Psi_y(-\infty)}{2}$.
Now substituting (\ref{sec4_v_eyy}) and (\ref{sec4_rhsterm3}) into (\ref{sec4_RHS1}), the right-hand side of (\ref{sec4_int2}) simplifies to
\begin{equation}
    \text{RHS}=-\frac{1}{c}\int \frac{w_c^3}{3}dy\left(\frac{\sqrt{\delta_1}b}{D_v}V_0-\frac{1}{c}\frac{V_0}{\sqrt{\delta_1}}\langle \Psi_y \rangle\right).
\end{equation}
Therefore, the small eigenvalue problem (\ref{sec4_int2}) reduces to
\begin{equation}\label{sec4_int3}
    \lambda\left(\int w_{cy}^2 dy-\frac{\tau}{c+\tau\lambda}\int w_cw_{cy}^2 dy\right)=\int \frac{w_c^3}{3}dy\left(\frac{1}{c}\langle \Psi_y \rangle-\frac{\delta_1b}{D_v}\right),
\end{equation}
in which the only term to determine is $\langle \Psi_y \rangle :=
 \frac{\Psi_y(\infty)+\Psi_y(-\infty)}{2}$. To evaluate this term, we consider the outer problem (\ref{sec4_small_2EP2}) for $\Psi$. 
 
 We have shown from (\ref{sec4_Phi0}) and (\ref{sec4_Phi1})  that 
 \begin{equation}
     \phi(x)=\Phi(y)\sim \Phi_0+\sqrt{\delta_1}\Phi_1=w_{cy}+\sqrt{\delta_1}\Psi_0(0)w_c+\sqrt{\delta_1}\Phi_{1,odd}.
 \end{equation}
 Note that $w_{cy}$ is like a dipole (odd) and $w_c$ behaves like a Dirac delta function, so taking $\sqrt{\delta_1}\ll \delta\ll 1$ and integrating (\ref{sec4_small_2EP2}) on $(-\delta,\delta)$ yields the following condition
\begin{equation}\label{sec4_outer_conditon1}
    \psi_x(\delta)- \psi_x(-\delta)\sim -\frac{2}{D_v} \int_{-\delta}^\delta U_e\phi dx\sim-\sqrt{\delta_1}\frac{2V_0}{cD_v}\Psi(0)\int w_c^2 dy.
\end{equation}
 
To derive the second jump condition for $\psi$, we multiply (\ref{sec4_small_2EP2}) by $x$ and integrate over $(-\delta,\delta)$. Applying integration by parts yields
\begin{equation}\label{sec4_outer_conditon2}
\psi(\delta)-\psi(-\delta)=-\frac{\sqrt{\delta_1}V_0}{cD_v}\int w_c^2dy
\end{equation}
Therefore, the outer variable $\psi(x)$ satisfies:
\begin{subequations} \label{sec4_psi_out}
\begin{align}
    &D_v\psi_{xx}-b\psi=0, \quad x\neq 0, \\
    &\psi(0^+)-\psi(0^-)=-\frac{\sqrt{\delta_1}V_0}{cD_v}
    \int w_c^2 dy,\\
    &\psi_x(0^+)-\psi_x(0^-)=-\frac{2\sqrt{\delta_1}V_0}{cD_v}\Psi_0(0)
    \int w_c^2 dy.
\end{align}
\end{subequations} 
Since (\ref{sec4_Psi_yy}) implies that $\Psi_{yy}$ is an odd function, so $\Psi=\Psi(0)+\text{odd function}$ and 
\begin{equation}
    \Psi(0)=\frac{\Psi(\infty)+\Psi(-\infty)}{2}.
\end{equation}
By matching $\psi(x)=\Psi(y)=\sqrt{\delta_1}\Psi_0(y)$ and introducing $\psi(x)=\sqrt{\delta_1}\eta(x)$, we get  
\begin{subequations} \label{sec4_eta_out}
\begin{align}
    &D_v\eta_{xx}-b\eta=0, \quad x\neq 0, \\
    &\eta(0^+)-\eta(0^-)=-\frac{V_{0}}{cD_v}
    \int w_c^2 dy,\\
    &\eta_x(0^+)-\eta_x(0^-)=-\frac{2V_0}{cD_v}\langle \eta\rangle
    \int w_c^2 dy.
\end{align}
\end{subequations} 
and the inner problem (\ref{sec4_int3}) written in terms of $\eta$ as
\begin{equation}\label{sec4_int4}
    \lambda\left(\int w_{cy}^2 dy-\frac{\tau}{c+\tau\lambda}\int w_cw_{cy}^2 dy\right)=\delta_1\int \frac{w_c^3}{3}dy\left(\frac{1}{c}\langle \eta_x \rangle-\frac{b}{D_v}\right),
\end{equation}
where $\langle \eta_x \rangle:=\frac{\eta_x(0^-)+\eta_x(0^+)}{2}$.  Now to solve $\eta$ from (\ref{sec4_int4}) we specify the boundary conditions of $\eta$ for single spike equilibrium as $\eta'(\pm L)=0$ and $\eta$ is odd. Therefore, $\eta_x(0^+)-\eta_x(0^-)=0$ and solving (\ref{sec4_eta_out}) yields
\begin{equation}\label{sec4_eta}
\eta(x)=-\frac{V_0}{cD_v}\frac{\int w_c^2 dy}{2\cosh\left(\sqrt{\frac{b}{D_v}}l\right)}\left\{
\begin{array}
[c]{ll}%
\cosh\left(\sqrt{\frac{b}{D_v}}(x+L)\right), & -l<x<0 \\
-\cosh\left(\sqrt{\frac{b}{D_v}}(x-L)\right), & 0<x<l.  \\
\end{array}
\right.
\end{equation}
We then calculate $\langle\eta_x \rangle=\frac{\eta_x(0^-)+\eta_x(0^+)}{2}=\frac{V_0}{2cD_v}\sqrt{\frac{b}{D_v}}\int w_c^2 dy\tanh\left(\sqrt{\frac{b}{D_v}}l\right)$. With $V_0=V_{0+}=\frac{\sqrt{bD_v}c^2}{3}\tanh\left(\sqrt{\frac{b}{D_v}}l\right)$, we obtain
\begin{equation}\label{sec4_eta_x}
    \langle\eta_x \rangle=\frac{bc}{6D_v}\int w_c^2 dy\tanh^2\left(\sqrt{\frac{b}{D_v}}l\right).
\end{equation}
Finally, plug (\ref{sec4_eta_x}) into the eigenvalue problem (\ref{sec4_int4}) and using the fact that $\int w_c^2 dy=6$, we get
\begin{equation}\label{sec4_int5}
\lambda\left(\int w_{cy}^2 dy-\frac{\tau}{c+\tau\lambda}\int w_cw_{cy}^2 dy\right)=-\delta_1\int \frac{w_c^3}{3}dy\frac{b}{D_v}\text{sech}^2\left(\sqrt{\frac{b}{D_v}}l\right)
\end{equation}

 As $\lambda=\mathcal{O}(\delta_1)\ll 1$, we approximate $\frac{\tau}{c+\tau \lambda}\sim \frac{\tau}{c} +\mathcal{O}(\delta_1)$ to get the leading order of $\lambda$
\begin{equation}\label{sec4_small_eig1}
    \lambda\sim\frac{-\delta_1\int \frac{w_c^3}{3}dy\frac{b}{D_v}\text{sech}^2\left(\sqrt{\frac{b}{D_v}}l\right)}{\int w_{cy}^2 dy-\frac{\tau}{c}\int w_cw_{cy}^2 dy}.
\end{equation}
It is obvious that the top of (\ref{sec4_small_eig1}) is always negative, therefore, using the fact that $\int w_{cy}^2 dy=\frac{6}{5}$ and $\int w_cw_{cy}^2 dy=\frac{36}{35}$, $\lambda$ crosses $0$ whenever
\begin{equation}\label{sec4_tau_h}
    \tau>\tau_h=\frac{c\int w_{cy}^2 dy}{\int w_cw_{cy}^2 dy}=\frac{7}{6}c.
\end{equation}

To obtain the full expression of $\lambda$, we rewrite (\ref{sec4_int5}) as the following quadratic equation
\begin{equation}\label{sec4_quadratic_lambda}
    \frac{7}{6}\tau\lambda^2-\left(\tau-\frac{7}{6}c-\frac{35}{36}\delta_1\tau k\right)\lambda+\frac{35}{36}\delta_1kc=0,
\end{equation}
where $k=\int \frac{w_c^3}{3}dy\frac{b}{D_v}\text{sech}^2\left(\sqrt{\frac{b}{D_v}}l\right)$. 

This implies that as $\tau$ increases over $\tau_h$, a pair of complex conjugate eigenvalues enter the unstable right half-plane, triggering an oscillatory instability in the motion of the spike, with
\begin{equation}\label{sec4_real_eig}
    Re(\lambda_{\pm})=\frac{\tau-\frac{7}{6}c-\frac{35}{36}\delta_1\tau k}{\frac{7}{3}\tau}\sim\frac{3}{7\tau}\left(\tau-\frac{7}{6}c\right),
\end{equation}
and \begin{equation}\label{sec4_imag_eig}
    Im(\lambda_{\pm})=\pm \sqrt{\delta_1\frac{5}{6}\frac{kc}{\tau}}i.
\end{equation}
We now summarize the discussion in the following. 
\begin{thm}\label{small_eig_threshold1}
    In the case $\tau>0, \theta=0$, the single-spike equilibrium of the system (\ref{sec3_Gm}) loses stability and undergoes a Hopf bifurcation as $\tau$ increases beyond $\tau_h\sim \frac{7}{6}c$.  Moreover, As $\tau\to \infty$, the imaginary part of the eigenvalue $Im(\lambda)\to 0$. 
\end{thm}

In Figure \ref{fig:small_eig1}, we compare asymptotic predictions $\tau_h$ with full numerical simulations of (\ref{sec3_Gm}).In the simulations, we record the critical value of $\tau$, beyond which the spike begins to oscillate periodically around the center of the domain. In terms of $c$, $\tau_h$ varies linearly and shows excellent agreement between analysis and simulations.  Moreover, as shown in the Introduction \ref{sec1}, Figure \ref{fig:seig_plot} presents full numerical simulations with FlexPDE illustrating the spike dynamics for values of $\tau$ exceeding the Hopf threshold $\tau_h$. For $\tau$ slightly larger than $\tau_h$, the interior spike undergoes small-amplitude oscillations around its equilibrium location, which implies the onset of Hopf instability. As $\tau$ is increased further, these oscillations grow in amplitude, and the spike begins to drift significantly. In this regime, the motion eventually drives the spike toward the boundary, indicating a transition from  oscillatory dynamics to drift-dominated motion, which is consistent with the decrease of the imaginary part of the eigenvalue as $\tau$ increases.

\begin{figure}[htbp]
    \centering   \includegraphics[width=0.6\textwidth]{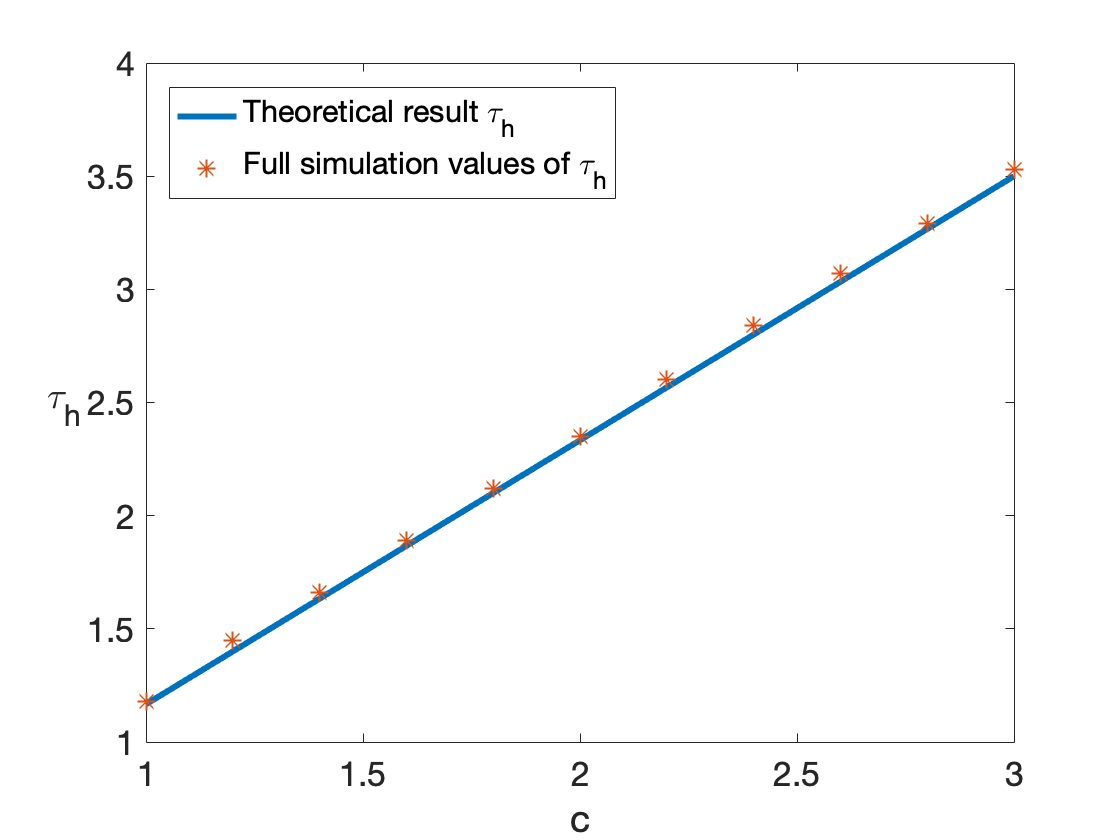}
    \caption{Comparison between asymptotic and numerical results for $\tau_h$ as parameter $c$ is varied. The solid curve is the asymptotic result given in \eqref{sec4_tau_h}. The stars are obtained by full simulations of the GM model \eqref{sec3_Gm} using \textit{Flexpde}. Parameters: $D_v=1, \delta_1 = 0.01^2, a=0.01,  b=1$ and $l = 1$.}
    \label{fig:small_eig1}%
\end{figure}

\section{Discussion}  \label{Sec5}

In this paper, we present an extension of the GM model (\ref{Gm}) in the semi-strong interaction regime, also characterized by an asymptotically large diffusivity ratio. Within this framework, we have constructed a single spike equilibrium for an arbitrary value of $a>0$.  From a mathematical perspective, the novelty of the analysis in contrast to previous studies \cite{alsaadi_2022, gai2020localized} is that we now must couple a nonlinear inner problem for the spike profile to a nonlinear reduced scalar boundary value problem (BVP) defined in the outer region away from the spike. This coupling leads to a more complex inner-outer interaction, but it applies to all values of $a>0$.

For the non-trivial background state $a>0$, we analyzed a global bifurcation mechanism that is responsible for the generation of spatial patterns as the inhibitor diffusivity $D_v$ decreases.  In parameter regimes where a one-spike solution exists on the infinite line, in Section \ref{Sec2: no nucleation} we showed that spike nucleation will not occur as $D_v$ decreases in the semi-strong interaction regime.  

By introducing the time-scaling parameter $\theta$ and $\tau$ in front of the equations for $v$ and $w$, respectively, we studied the novel behavior introduced by the third component $w$. The dynamics now exhibits not only large-scale oscillatory motion in the amplitude, which is triggered by large-eigenvalue instabilities, but also oscillatory spike motion associated with small-eigenvalues crossing into the right half-plane.  These two mechanisms highlight a key difference from classical two-component RD systems \cite{ward2003hopf, alsaadi_2022, gai2023police}, as well as some three-component framework \cite{al2024localized}, where only large-scale oscillatory motion in the amplitude is observed. As a result, the extended model supports a richer variety of oscillatory dynamics.

There are numerous open questions for future study.  In Section \ref{Sec3} we have derived a novel nonlocal eigenvalue problem (NLEP) due to the presence of the $\tau$-dependent term, it would be interesting to provide a rigorous study for the spectrum of NLEP.  Previous work on NLEPs with eigenvalue dependence has focused primarily on cases where the eigenvalue enters the non-local term rather than the operator itself \cite{ward2003hopf}. A systematic investigation of the spectrum in the present three-component setting could reveal new bifurcation structures.

In this paper, we have studied the effects of $\theta$ and $\tau$ separately under the regime $a\ll 1$.  It would be natural to extend the stability analysis to the case $a=\mathcal{O}(1)$.  Preliminary simulations suggest that this setting gives rise to a rich variety of spike dynamics that deserve further study, including spike motion, spike nucleation, and spike competition leading to spike death.  It would therefore be interesting to investigate the detailed spike dynamics in this regime, as well as the possible interplay with multi-spike patterns.

\begin{figure}
\begin{center}
\includegraphics[width=0.48\textwidth]{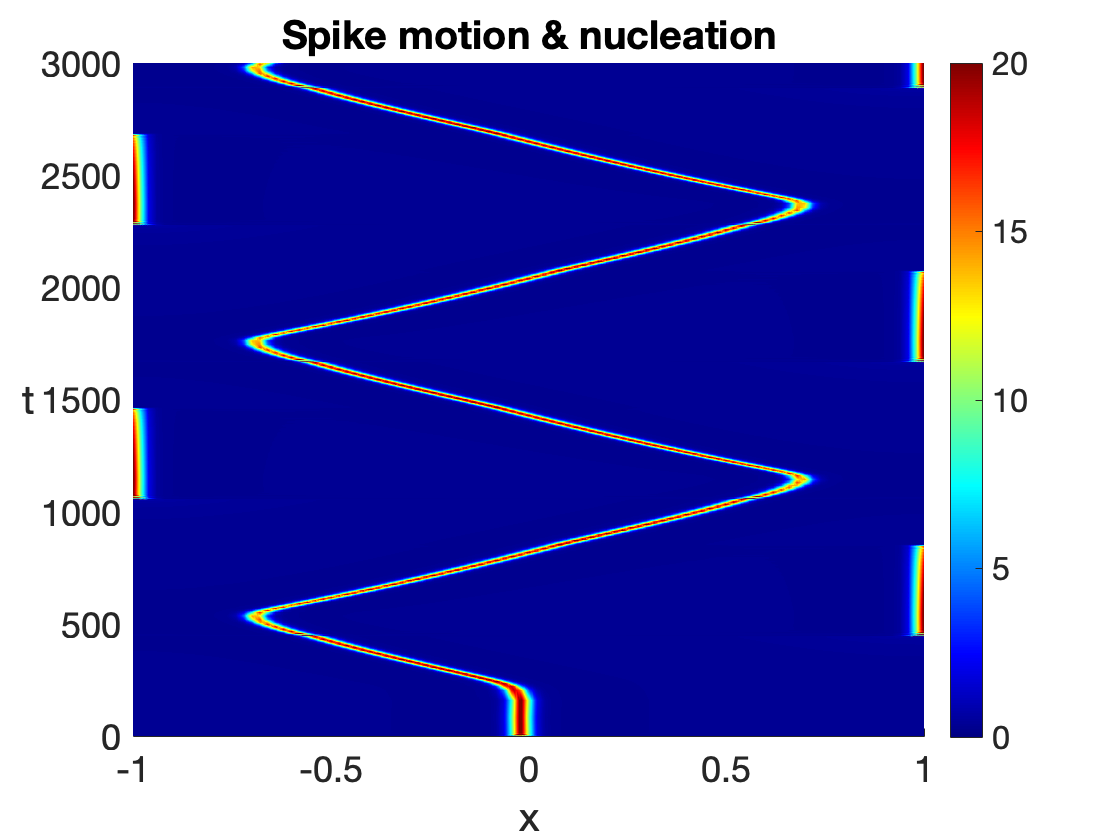} \quad
\includegraphics[width=0.48\textwidth]{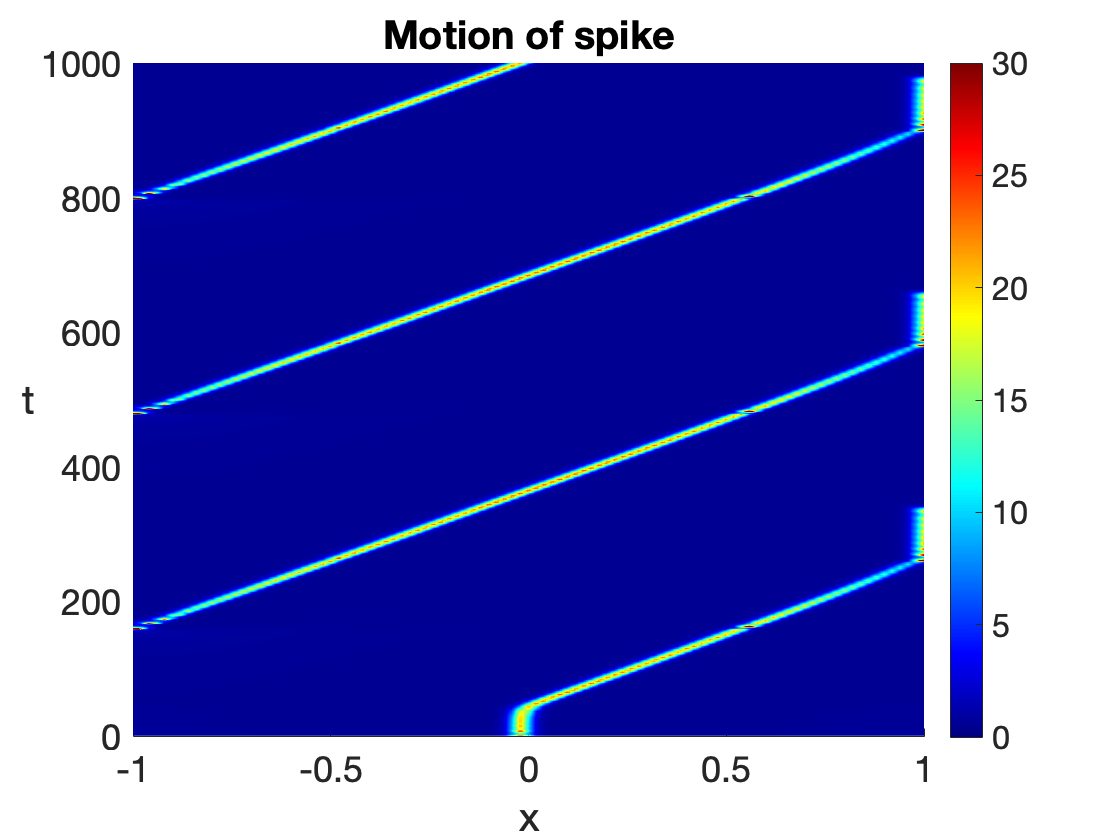} \newline
(a) ~~~~~~~~~~~~~~~~~~~~~~~~~~~~~~~~~~~~~~~~~~~~~~~~~~~~~~~~~~~~ (b)
\caption{Full simulations by {Flexpde} \cite{flexpde2015} illustrating spike motions for different values of $\tau$ triggering spike nucleation and annihilation dynamics; 
(a)For $\tau=1.25$, the interior spike exhibits oscillations that eventually induce boundary annihilation and nucleation on the opposite side. (b) $\tau=1.5$, the spike drifts toward the boundary, triggering nucleation at the opposite boundary and subsequent spike competition. Other parameters are: $\delta_1=0.01^2, a=0.3, b=1, c=1, l=1,\theta=0, D_v=0.2.$}
\label{fig:open1}
\end{center}
\end{figure}

An example is shown in Figure \ref{fig:open1} for $a>0$, and $\tau$ is sufficiently large. In the left panel, the spike motion leads to the nucleation as the spike radius near the far boundary becomes large enough. However, as the spike oscillates back, this motion induces competition: the boundary spike is annihilated, while a new spike is nucleated on the opposite side. In the right panel, when $\tau$ is increased further, the interior spike no longer oscillates but instead drifts toward the boundary, and triggers nucleation at the opposite boundary. This newly generated boundary spike then transitions into an interior spike, which in turn annihilates the previously existing interior spike once it becomes a boundary.

%coexistence of amplitude oscillations and spike motion. This induces complex dynamics phenomenon. Moreover, near the Hopf threshold $\theta_h$ and $\tau_h$, oscillations emerge either in spike amplitudes or spike motions. It would therefore be interesting to further study the detailed dynamics of a single spike in this regime, as well as the possible interplay with multi-spike patterns.

\section*{Acknowledgments}
Chunyi Gai gratefully acknowledges the support of the NSERC Discovery Grant Program.
The work of Fahad Al Saadi is funded by the Ministry of Higher Education, Research, and Innovation (MoHERI) under the Block Funding Program and conducted with support from the Military Technological College (MTC), Oman.

\appendix
\section*{Appendix: The Local Eigenvalue Problem \texorpdfstring{$L_\lambda\phi=\lambda\phi$}{Llambda phi = lambda phi}} \label{app}
\addcontentsline{toc}{section}{Appendix} % optional: add to ToC

\renewcommand{\theequation}{\arabic{equation}}
\renewcommand{\theequation}{A.\arabic{equation}}
\setcounter{equation}{0}

In this appendix, we solve the local eigenvalue problem (\ref{sec3_L_lambda_eig}) as follows.
\begin{equation}\label{app_local_problem}
    L_\lambda \phi:= \phi_{yy}-\phi+\left(2+\frac{\tau\lambda}{c+\tau\lambda}\right)w_c\phi=\lambda \phi,
\end{equation}
where $w_c=\frac{3}{2}\text{sech}^2(\frac{y}{2}).$
We are interested in finding any positive solution $\lambda>0$.
Let $z=\frac{x}{2}$, then (\ref{app_local_problem}) becomes
\begin{equation}\label{app_local_scaled}
\phi_{zz}+\left[\left(12+6\frac{\tau\lambda}{c+\tau\lambda}\right)\text{sech}^2(z)-4(1+\lambda)\right]\phi=0,
\end{equation}
which is of the well-known P\"oschl Teller type, and has known exact eigenvalues\cite{flugge2012practical}. In particular, in the canonical form
\begin{equation}\label{app_Poschl--Teller }
    \phi_{zz}+\left(p(p+1)\text{sech}^2(\frac{y}{2})-k^2\right)\phi=0, \quad v>0.
\end{equation}
The parameter $p$ controls the strength of the potential bump, and there are $[p]$ localized eigenfunctions (here $[p]$ denotes the floor of $p$), which can be expressed in terms of associated Legendre or hypergeometric functions. Their discrete levels are
\begin{equation}\label{app_kn}
    k_n^2=(v-n)^2, \quad n=0,1,..[v]-1.
\end{equation}
Now we match the parameters in our local eigenvalue problem (\ref{app_local_problem}), which yields the following system.
\begin{equation}
    \Psi(0)=\frac{\Psi(\infty)+\Psi(-\infty)}{2}.
\end{equation}
By matching $\psi(x)=\Psi(y)=\sqrt{\delta_1}\Psi_0(y)$ and introducing $\psi(x)=\sqrt{\delta_1}\eta(x)$, we get  
\begin{subequations} \label{app_match}
\begin{align}
   v(v+1)&=12+6\frac{\tau\lambda}{c+\tau\lambda}, \\
    k^2&=4(1+\lambda),\\
    k_n^2&=(v-n)^2.
\end{align}
\end{subequations} 
Solving the system (\ref{app_match}), we find that $\lambda$ satisfies 
\begin{equation}\label{app_lambda_eqn}
    F(\lambda):= \frac{\left(p(\lambda)-n\right)^2}{4}-1-\lambda=0,\quad n=0,1,..[p]-1,
\end{equation}
where
\begin{equation}\label{app_p(lambda)}
    p(\lambda)=\frac{-1+\sqrt{1+4\left(12+6\frac {\tau\lambda}
    {\tau\lambda+c}\right)}}{2}.
\end{equation}
\ref{app_lambda_eqn} can be further reduced to the following equation,
\begin{equation}\label{app_lambda_eqn2}
    12+6\frac {\tau\lambda}
    {\tau\lambda+c}=n^2+n+4+4\lambda+\sqrt{n^4-n^2+(8+8\lambda)\left(2n^2+2n+\frac{1}{2}\right)}.
\end{equation}
As we are only interested in positive solutions and $\lambda$ decreases with the mode index $n$. In fact, for $n=1$, (\ref{app_lambda_eqn2}) reduces to 
\begin{equation}
    1+\frac{\tau\lambda}{c+\tau\lambda}-\frac{2}{3}\lambda-\sqrt{1+\lambda}=0,
\end{equation}
which always admits the solution $\lambda=0$ for an arbitrary value of $c$ and $\tau$, so we seek a positive solution for the lowest mode $n=0$, which satisfies
\begin{equation}\label{app_thm3}
    4+3\frac{\tau\lambda}{c+\tau\lambda}-2\lambda-\sqrt{1+\lambda}=0.
\end{equation}
Moreover, as $\tau\to \infty$, equation (\ref{app_thm3}) reduces to 
\begin{equation}
    7-2\lambda-\sqrt{1+\lambda}=0,
\end{equation}
which has a solution $\lambda\sim 2.56$.

%with a focus on the regime in the limit $a \ll 1 $, so that $\gamma\sim \frac{ac\sqrt{\delta_1}}{V_0}\ll 1$, and the localized spike equilibrium can be expressed in the following closed form: 
%\begin{equation}\label{sec3_ss_a_small}
   % u_e\sim \frac{V_{0}}{c\sqrt{\delta_1}}w_c\left(\frac{x}{\sqrt{\delta_1}}\right), \quad v_e\sim \frac{a^2}{b}+\left(\frac{V_{0}}{\sqrt{\delta_1}}-\frac{a^2}{b}\right)\frac{\cosh\left(\sqrt{b}(l-|x|)\sqrt{D_v}\right)}{\cosh(\sqrt{b}l/\sqrt{D_v})}, \quad w_e=\frac{u_e}{c},
%\end{equation}
%where $V_{0}$ has two values given in (\ref{sec2_V0_roots_a_small}), and 
%\begin{equation}
  %  w_c(y)=\frac{3}{2}\text{sech}^2(y/2).
%\end{equation}. 

\bibliographystyle{plain} 
\bibliography{manuscript} 

\begin{thebibliography}{10}

\bibitem{alsaadi_2022}
F.~{Al Saadi}, A.R. Champneys, C.~Gai, and T.~Kolokolnikov.
\newblock Spikes and localised patterns for a novel schnakenberg model in the semi-strong interaction regime.
\newblock {\em European Journal of Applied Mathematics}, 33(1):133–152, 2022.

\bibitem{al2024localized}
Fahad Al~Saadi, Chunyi Gai, and Mark Nelson.
\newblock Localized pattern formation: semi-strong interaction asymptotic analysis for three components model.
\newblock {\em Proceedings of the Royal Society A}, 480(2281):20230591, 2024.

\bibitem{flexpde2015}
PDE FlexPDE.
\newblock Solutions inc.
\newblock {\em URL http://www. pdesolutions. com}, 2015.

\bibitem{flugge2012practical}
Siegfried Fl{\"u}gge.
\newblock {\em Practical quantum mechanics}.
\newblock Springer Science \& Business Media, 2012.

\bibitem{gai2020localized}
Chunyi Gai, David Iron, and Theodore Kolokolnikov.
\newblock Localized outbreaks in an sir model with diffusion.
\newblock {\em Journal of Mathematical Biology}, 80(5):1389--1411, 2020.

\bibitem{gai2025asymptotic}
Chunyi Gai, Edgardo Villar-Sep{\'u}lveda, Alan Champneys, and Michael~J Ward.
\newblock An asymptotic analysis of spike self-replication and spike nucleation of reaction-diffusion patterns on growing 1-d domains.
\newblock {\em Bulletin of Mathematical Biology}, 87(4):48, 2025.

\bibitem{gai2023police}
Chunyi Gai and Michael Ward.
\newblock The nucleation-annihilation dynamics of hotspot patterns for a reaction-diffusion system of urban crime with police deployment.
\newblock submitted to SIADS (37 pages), 2023.

\bibitem{gierer1972theory}
Alfred Gierer and Hans Meinhardt.
\newblock A theory of biological pattern formation.
\newblock {\em Kybernetik}, 12(1):30--39, 1972.

\bibitem{iron2001stability}
David Iron, Michael~J Ward, and Juncheng Wei.
\newblock The stability of spike solutions to the one-dimensional gierer--meinhardt model.
\newblock {\em Physica D: Nonlinear Phenomena}, 150(1-2):25--62, 2001.

\bibitem{kolokolnikov2006stability}
Theodore Kolokolnikov, Wentao Sun, Michael Ward, and Juncheng Wei.
\newblock The stability of a stripe for the gierer--meinhardt model and the effect of saturation.
\newblock {\em SIAM Journal on Applied Dynamical Systems}, 5(2):313--363, 2006.

\bibitem{kolokolnikov2009existence}
Theodore Kolokolnikov, Juncheng Wei, and Matthias Winter.
\newblock Existence and stability analysis of spiky solutions for the gierer--meinhardt system with large reaction rates.
\newblock {\em Physica D: Nonlinear Phenomena}, 238(16):1695--1710, 2009.

\bibitem{meinhardt2008models}
Hans Meinhardt.
\newblock Models of biological pattern formation: from elementary steps to the organization of embryonic axes.
\newblock {\em Current topics in developmental biology}, 81:1--63, 2008.

\bibitem{murray2003spatial}
James~Dickson Murray.
\newblock Spatial models and biomedical applications.
\newblock {\em Mathematical Biology}, 2003.

\bibitem{piskovsky2025turing}
Vit Piskovsky.
\newblock Turing instabilities for three interacting species.
\newblock {\em Applied Mathematics Letters}, 159:109269, 2025.

\bibitem{satnoianu2000turing}
Razvan~A Satnoianu, Michael Menzinger, and Philip~K Maini.
\newblock Turing instabilities in general systems.
\newblock {\em Journal of mathematical biology}, 41(6):493--512, 2000.

\bibitem{turing}
A.~Turing.
\newblock The chemical basis of morphogenesis.
\newblock {\em Phil. Trans. Roy. Soc. London, B}, 237:37--72, 1952.

\bibitem{pde2path}
H.~Uecker.
\newblock {\em Numerical continuation and bifurcation in Nonlinear PDEs}.
\newblock SIAM, 2021.

\bibitem{ward2002existence}
Michael~J Ward and Juncheng Wei.
\newblock The existence and stability of asymmetric spike patterns for the schnakenberg model.
\newblock {\em Studies in Applied Mathematics}, 109(3):229--264, 2002.

\bibitem{ward2003hopf}
Michael~Jeffrey Ward and Juncheng Wei.
\newblock Hopf bifurcations and oscillatory instabilities of spike solutions for the one-dimensional gierer-meinhardt model.
\newblock {\em Journal of Nonlinear Science}, 13(2), 2003.

\bibitem{wei1999single}
Juncheng Wei.
\newblock On single interior spike solutions of the gierer--meinhardt system: uniqueness and spectrum estimates.
\newblock {\em European Journal of Applied Mathematics}, 10(4):353--378, 1999.

\bibitem{wei2007existence}
Juncheng Wei and Matthias Winter.
\newblock Existence, classification and stability analysis of multiple-peaked solutions for the gierer-meinhardt system in r1.
\newblock 2007.

\bibitem{wei2008mutually}
Juncheng Wei and Matthias Winter.
\newblock Mutually exclusive spiky pattern and segmentation modeled by the five-component meinhardt--gierer system.
\newblock {\em SIAM Journal on Applied Mathematics}, 69(2):419--452, 2008.

\bibitem{wei2013mathematical}
Juncheng Wei and Matthias Winter.
\newblock {\em Mathematical aspects of pattern formation in biological systems}, volume 189.
\newblock Springer Science \& Business Media, 2013.

\bibitem{xie2021complex}
Shuangquan Xie, Theodore Kolokolnikov, and Yasumasa Nishiura.
\newblock Complex oscillatory motion of multiple spikes in a three-component schnakenberg system.
\newblock {\em Nonlinearity}, 34(8):5708, 2021.

\bibitem{xie2024oscillatory}
Shuangquan Xie, Wen Yang, and Jiaojiao Zhang.
\newblock Oscillatory motions of multiple spikes in three-component reaction--diffusion systems.
\newblock {\em Journal of Nonlinear Science}, 34(4):78, 2024.

\end{thebibliography}
\end{document}